\documentclass[final,leqno]{siamltex}

\usepackage{amsmath,amssymb,amscd,amsxtra,amsfonts}
\usepackage{epsf,graphicx,epsfig,enumitem}

\newtheorem{problem}[theorem]{Problem}

\title{inverse obstacle scattering for Maxwell's equations in an
unbounded structure}

\author{Peijun Li\thanks{Department of Mathematics, Purdue University, West
Lafayette, Indiana 47907, USA ({\tt lipeijun@math.purdue.edu}).} \and Jue
Wang\thanks{School of Science, Harbin Engineering University, Harbin, 150001,
China ({\tt wangjue3721@163.com}). The research was supported in part by a
National Natural Science Foundation of China (No. 11801116) and Fundamental
Research Funds for the Central Universities (No. GK2110260213).} \and Lei
Zhang\thanks{School of Mathematical Sciences, Heilongjiang University, Harbin,
150080, China ({\tt zl19802003@163.com}). The work was supported partially
by a National Natural Science Foundation of China (No. 11871198), the Special
Funds of Science and Technology Innovation Talents of Harbin (No. 2017RAQXJ099)
and the Fundamental Research Funds for the Universities of Heilongjiang
Province - Heilongjiang University Special Fund Project (No. RCYJTD201804).}}

\begin{document}

\maketitle

\begin{abstract}
This paper is concerned with analysis of electromagnetic wave scattering by an
obstacle which is embedded in a two-layered lossy medium separated by an
unbounded rough surface. Given a dipole point source, the
direct problem is to determine the electromagnetic wave field for the
given obstacle and unbounded rough surface; the inverse problem is to
reconstruct simultaneously the obstacle and unbounded rough surface from the
electromagnetic field measured on a plane surface above the obstacle. For
the direct problem, a new boundary integral equation is proposed and its
well-posedness is established. The analysis is based on the exponential decay of
the dyadic Green function for Maxwell's equations in a lossy medium. For the
inverse problem, the global uniqueness is proved and a local stability
is discussed. A crucial step in the proof of the stability is to obtain the
existence and characterization of the domain derivative of the electric field
with respect to the shape of the obstacle and unbounded rough surface.
\end{abstract}

\begin{keywords}
Maxwell's equations, inverse scattering problem, unbounded rough
surface, domain derivative, uniqueness, local stability
\end{keywords}

\begin{AMS}
78A46, 78M30
\end{AMS}

\pagestyle{myheadings}
\thispagestyle{plain}
\markboth{P. Li, J. Wang, and L. Zhang}{Obstacle Scattering for Maxwell's
Equations}

\section{Introduction}

Consider the electromagnetic scattering of a dipole point source illumination by
an obstacle which is embedded in a two-layered medium separated by an unbounded
rough surface in three dimensions. An obstacle is referred to as an impenetrable
medium which has a bounded closed surface; an unbounded rough surface stands for
a nonlocal perturbation of an infinite plane surface such that the perturbed
surface lies within a finite distance of the original plane. Given the dipole
point source, the direct problem is to determine the electromagnetic wave field
for the known obstacle and unbounded rough surface; the inverse problem is to
reconstruct both of the obstacle and the unbounded rough surface, from the
measured wave field. The scattering problems arise from diverse scientific
areas such as radar and sonar, geophysical exploration, nondestructive testing,
and medical imaging. In particular, the obstacle scattering in unbounded
structures has significant applications in radar based object recognition above
the sea surface and detection of underwater or underground mines.

As a fundamental problem in scattering theory, the obstacle scattering problem,
where the obstacle is embedded in a homogeneous medium, has been examined
extensively by numerous researchers. The details can be found in the monographs
\cite{CK-83, N-00} and \cite{CC-05, CK-98, KG-08} on the mathematical and
numerical studies of the direct and inverse problems, respectively. The
unbounded rough surface scattering problems have also been widely examined in
both of the mathematical and engineering communities. We refer to
\cite{DM-JASA97, EG-WRM04, HLS-JCP14, M-JASA91, O-91, SS-WRM01, V-94, WC-WRM01,
ZMW-WM13} for various solution methods including mathematical, computational,
approximate, asymptotic, and statistical methods. The scattering problems in
unbounded structures are quite challenging due to two major issues: the usual
Silver--M\"{u}ller radiation condition is no longer valid; the Fredholm
alternative argument does not apply due to the lack of compactness result. The
mathematical analysis can be found in \cite{CZ-SIAP98, CM-SIMA05, LR-JAM10,
LS-M2AS12, ZC-M2AS03} and \cite{HL-SIMA11, LWZ-SIMA11, LZZ-M2AS16} on the
well-posedness of the two-dimensional Helmholtz equation and the
three-dimensional Maxwell equations, respectively. The inverse problems have
also been considered mathematically and computationally for unbounded rough
surfaces in \cite{BL-SJIS14, BZ-IP16, BZ-SJIS18, LZZ-SJIS18}.

In this paper, we study the electromagnetic obstacle scattering for the
three-dimensional Maxwell equations in an unbounded structure. Specifically, we
consider the illumination of a time-harmonic electromagnetic wave, generated
from a dipole point source, onto a perfectly electrically conducting obstacle
which is embedded in a two-layered medium separated by an unbounded rough
surface. The obstacle is located either above or below the surface and may have
multiple disjoint components. For simplicity of presentation, we assume that the
obstacle has only one component and is located above the surface. The free
spaces are assumed to be filled with some homogeneous and lossy materials
accounting for the energy absorption. The problem has received much attention
and many computational work have been done in the engineering community
\cite{J-OTL01, KB-TAP11, LGJW-GRS12}. However, the rigorous analysis is very
rare, especially for the three-dimensional Maxwell equations.

In this work, we introduce an energy decaying condition to replace the
Silver--M\"{u}ller radiation condition in order to ensure the uniqueness of the
solution. The asymptotic behaviour of dyadic Green's function is analyzed and
plays an important role in the analysis for the well-posedness of the direct
problem. A new boundary integral equation is proposed for the associated
boundary value problem. Based on some energy estimates, the uniqueness of the
solution for the scattering problem is established. For the inverse problem, we
intend to answer the following question: what information can we extract about
the obstacle and the unbounded rough surface from the tangential trace of the
electric field measured on the plane surface above the obstacle? The first
result is a global uniqueness theorem. We show that any two obstacles and
unbounded rough surfaces are identical if they generate the same data. The proof
is based on a combination of the Holmgren uniqueness, unique continuation, and a
construction of singular perturbation. The second result is concerned with a
local stability: if two obstacles are ``close'' and two unbounded rough surfaces
are also ``close'', then for any $\delta>0$, the measurements of the two
tangential trace of the electric fields being $\delta$-close implies that
both of the two obstacles and the two unbounded rough surfaces are $\mathcal
O(\delta)$-close. A crucial step in the stability proof is to obtain the
existence and characterization of the domain derivative of the electric field
with respect to the shape of the obstacle and unbounded rough surface.

The paper is organized as follows. In Section 2, we introduce the model problem
and present some asymptotic analysis for dyadic Green's function of the Maxwell
equations. Section 3 is devoted to the well-posedness of the direct scattering
problem. An equivalent integral representation is proposed for the boundary
value problem. A new boundary integral equation is developed and its
well-posedness is established. In Sections 4 and 5, we discuss the global
uniqueness and local stability of the inverse problem, respectively. The domain
derivative is studied. The paper is concluded with some general remarks in
Section 6.

\section{Problem formulation}

Let us first specify the problem geometry which is shown in Figure \ref{pg}.
Let $S$ be an unbounded rough surface given by
\[
S=\{\boldsymbol x=(x_{1},x_{2},x_{3})\in \mathbb{R}^{3}: x_{3}=f(x_{1},
x_{2})\},
\]
where $f\in C^2(\mathbb R^2)$. The surface $S$ divides $\mathbb R^3$
into $\Omega_1^+$ and $\Omega_2$, where
\begin{equation*}
\Omega_{1}^{+}=\{\boldsymbol{x}\in \mathbb{R}^{3}: x_{3}>f(x_{1}, x_{2})\},
\quad
\Omega_{2}=\{\boldsymbol{x}\in \mathbb{R}^{3}: x_{3}<f(x_{1},x_{2})\}.
\end{equation*}
Let $D$ be a bounded obstacle with $C^2$ boundary $\Gamma$. The obstacle is
assumed to be a perfect electrical conductor which is located either in
$\Omega_1^+$ or in $\Omega_2$. For instance, we may assume that
$D\subset\subset\Omega_1^+$. Define $\Omega_1=\Omega_1^+\setminus\overline D$.
The domain $\Omega_j$ is assumed to be filled with some homogeneous, isotropic,
and absorbing medium which may be characterized by the dielectric permittivity
$\varepsilon_j>0$, the magnetic permeability $\mu_j>0$, and the electric
conductivity $\sigma_j>0$, $j=1, 2$.

\begin{figure}
\centering
\includegraphics[width=0.6\textwidth]{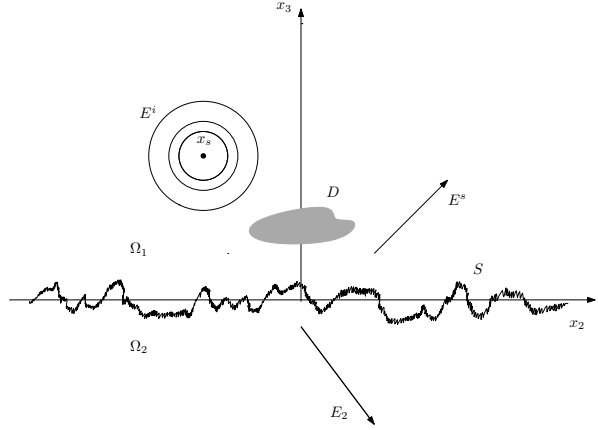}
\caption{Problem geometry of the obstacle scattering in an unbounded structure.}
\label{pg}
\end{figure}

In $\Omega_j$, the electromagnetic waves satisfy the time-harmonic Maxwell
equations (time dependence $e^{-\mathrm{i}\omega t}$):
\begin{equation*}\label{EHeq}
\begin{cases}
\nabla\times\boldsymbol{E}_{j}=\mathrm{i}\omega\mu_{j}\boldsymbol{H}_{j},\\
\nabla\times\boldsymbol{H}_{j}=-\mathrm{i}\omega\varepsilon_{j}\boldsymbol{E}_{j
}+\boldsymbol{J}_{j},\\
\nabla\cdot(\varepsilon_{j}\boldsymbol{E}_{j})=\rho_{j},\\
\nabla\cdot(\mu_{j}\boldsymbol{H}_{j})=0,
\end{cases}
\end{equation*}
where $\omega>0$ is the angular frequency, $\boldsymbol{E}_{j}$,
$\boldsymbol{H}_{j}$, $\boldsymbol{J}_{j}$ denote the electric field, the
magnetic field, the electric current density, respectively, and
$\rho_{j}=({\rm i}\omega)^{-1}\nabla\cdot\boldsymbol{J}_{j}$ is the
electric charge density. The external current source is assumed to be located in
$\Omega_1$. The relation between the electric current density and the electric
field is given by
\begin{align*}
\begin{cases}
\boldsymbol{J}_{1}=\sigma_{1}\boldsymbol{E}_{1}+\boldsymbol{J}_{cs} &\quad
\text{in} ~ \Omega_{1},\\
\boldsymbol{J}_{2}=\sigma_{2}\boldsymbol{E}_{2} &\quad
\text{in} ~ \Omega_{2},
\end{cases}
\end{align*}
where $\boldsymbol{J}_{cs}$ stands for the current source.

Using the above constitutive relation, we obtain coupled systems
\begin{equation}\label{EHeq1}
\begin{cases}
\nabla\times\boldsymbol{E}_{1}=\mathrm{i}\omega\mu_{1}\boldsymbol{H}_{1},\\
\nabla\times\boldsymbol{H}_{1}=-\mathrm{i}\omega\left(\varepsilon_{1}+\mathrm{i}
\frac{\sigma_{1}}{\omega}\right)\boldsymbol{E}_{1}+\boldsymbol{J}_{cs},\\
\left(\varepsilon_{1}+\mathrm{i}\frac{\sigma_{1}}{\omega}
\right)\nabla\cdot\boldsymbol{E}_{1}=\frac{1}{\mathrm{i}\omega}
\nabla\cdot\boldsymbol{J}_{cs},\\
\nabla\cdot(\mu_{1}\boldsymbol{H}_{1})=0,
\end{cases}
\quad\text{in}~\Omega_1,
\end{equation}
and
\begin{equation}\label{EHeq2}
\begin{cases}
\nabla\times\boldsymbol{E}_{2}=\mathrm{i}\omega\mu_{2}\boldsymbol{H}_{2},\\
\nabla\times\boldsymbol{H}_{2}=-\mathrm{i}\omega\left(\varepsilon_{2}+\mathrm{i}
\frac{\sigma_{2}}{\omega}\right)\boldsymbol{E}_{2}, \quad \quad \quad\\
\left(\varepsilon_{2}+\mathrm{i}\frac{\sigma_{2}}{\omega}
\right)\nabla\cdot\boldsymbol{E}_{2}=0,\\
\nabla\cdot(\mu_{2}\boldsymbol{H}_{2})=0,
\end{cases}
\quad\text{in}~ \Omega_2.
\end{equation}
Eliminating the magnetic field $\boldsymbol H_1$ in \eqref{EHeq1}, we obtain a
decoupled equation for the electric field $\boldsymbol E_1$:
\begin{equation}\label{Eeq1}
\nabla\times(\nabla\times\boldsymbol{E}_{1}(\boldsymbol{x}))
-\kappa_{1}^{2}\boldsymbol{E}_{1}(\boldsymbol{x})
=\mathrm{i}\omega\mu_{1}\boldsymbol{J}_{cs}(\boldsymbol{x}), \quad
\boldsymbol{x}\in\ \Omega_{1}.
\end{equation}
Similarly, it follows from \eqref{EHeq2} that we may deduce a decoupled Maxwell
system for the electric field $\boldsymbol E_2$:
\begin{equation}\label{Eeq2}
\nabla\times(\nabla\times\boldsymbol{E}_{2}(\boldsymbol{x}))
-\kappa_{2}^{2}\boldsymbol{E}_{2}(\boldsymbol{x})=0,\quad \boldsymbol{x}\in\
\Omega_{2}.
\end{equation}
Here $\kappa_{j}=\omega\sqrt{\left(\varepsilon_{j}+\mathrm{i}\frac{\sigma_{j}}{
\omega }\right)\mu_{j}}$ is the wave number in $\Omega_j, j=1, 2$.
Since $\varepsilon_j, \mu_j, \sigma_j$ are positive constants, $\kappa_j$ is a
complex constant with $\Re\kappa_j>0, \Im\kappa_j>0,$ which accounts for the
energy absorption.

By the perfect conductor assumption for the obstacle, it holds that
\begin{equation}\label{EBCS1}
\boldsymbol{\nu}_{\Gamma}\times{\boldsymbol{E}_{1}}=0\quad \text{on} ~ \Gamma,
\end{equation}
where $\boldsymbol{\nu}_{\Gamma}$ denotes the unit normal vector on the boundary
$\Gamma$ directed into the exterior of $D$. The usual continuity conditions
need to be imposed, i.e., the tangential traces of the electric and magnetic
fields are continuous across $S$:
\begin{equation}\label{EBCS}
\boldsymbol{\nu}_{S}\times{\boldsymbol{E}_{1}}=\boldsymbol{\nu}_{S}\times{
\boldsymbol{E}_{2}},
\quad \boldsymbol{\nu}_{S}\times{\boldsymbol{H}_{1}}=\boldsymbol{\nu}_{S}\times{
\boldsymbol{H}_{2}} \quad \text{on} ~ S,
\end{equation}
where $\boldsymbol{\nu}_{S}$ denotes the unit normal vector on $S$ pointing from
$\Omega_{2}$ to $\Omega_{1}$.

The incident electromagnetic fields
$(\boldsymbol{E}^{i}, \boldsymbol{H}^{i})$ satisfy Maxwell's equations
\begin{equation}\label{Eieq}
\begin{cases}
\nabla\times(\nabla\times\boldsymbol{E}^{i}(\boldsymbol{x}))
-\kappa_{1}^{2}\boldsymbol{E}^{i}(\boldsymbol{x})=\mathrm{i}\omega\mu_{1}
\boldsymbol{J}_{cs}(\boldsymbol{x}),\\
\nabla\times(\nabla\times\boldsymbol{H}^{i}(\boldsymbol{x}))
-\kappa_{1}^{2}\boldsymbol{H}^{i}(\boldsymbol{x})=\nabla\times\boldsymbol{J}_{cs
}(\boldsymbol{x}),
\end{cases}
\quad \boldsymbol{x}\in\Omega_1.
\end{equation}
In $\Omega_1$, the total electromagnetic fields
$(\boldsymbol{E}_{1},\boldsymbol{H}_{1})$
consist of the incident fields $(\boldsymbol{E}^{i},\boldsymbol{H}^{i})$ and the
scattered fields $(\boldsymbol{E}^{s},\boldsymbol{H}^{s})$. In $\Omega_{2}$, the
electromagnetic
fields $(\boldsymbol{E}_{2}(\boldsymbol{x}),\boldsymbol{H}_{2}(\boldsymbol{
x}))$ are called the transmitted fields.

In addition, we propose an energy decaying condition
\begin{equation}\label{EsIRC}
\lim_{r\to +\infty}\int_{\partial
B_{r}^{+}}|\boldsymbol{E}^{s}|^{2}{\rm d}s=0, \quad
\lim_{r\to +\infty}\int_{\partial
B_{r}^{+}}|\boldsymbol{H}^{s}|^{2}{\rm d}s=0
\end{equation}
and
\begin{equation}\label{E2IRC}
\lim_{r\to +\infty}\int_{\partial
B_{r}^{-}}|\boldsymbol{E}_{2}|^{2}{\rm d}s=0, \quad
\lim_{r\to +\infty}\int_{\partial
B_{r}^{-}}|\boldsymbol{H}_{2}|^{2}{\rm d}s=0,
\end{equation}
where $\partial B_{r}^{\pm}$ denotes the hemisphere of radius $r$ above or
below $S$.

The dyadic Green function is defined by the solution of the following equation
\begin{equation}\label{G1eq}
\nabla_{\boldsymbol{x}}\times(\nabla_{\boldsymbol{x}}\times\boldsymbol{G}_{j}
(\boldsymbol{x}-\boldsymbol{y}))
-\kappa_{j}^{2}\boldsymbol{G}_{j}(\boldsymbol{x} -\boldsymbol{y})
=\delta(\boldsymbol{x}-\boldsymbol{y})\boldsymbol{I}\quad \text{in} ~
\Omega_{j},
\end{equation}
where $\boldsymbol{I}$ is the unitary dyadic and $\delta$ is the Dirac delta
function. It is known that the dyadic Green function is given by
\begin{equation}\label{G}
\boldsymbol{G}_{j}(\boldsymbol{x}-\boldsymbol{y})
=\bigg[\boldsymbol{I}+\frac{
\nabla_{\boldsymbol{y}}\nabla_{\boldsymbol{y}}}{\kappa_{j}^{2}}\bigg]
\frac{\exp{(\mathrm{i}\kappa_{j}|\boldsymbol{x}-\boldsymbol{y}|)}}{
4\pi|\boldsymbol{x} -\boldsymbol{y}|}.
\end{equation}

We assume that the dipole point source is located at $\boldsymbol
x_s\in\Omega_1$ and has a polarization $\boldsymbol{q}\in
\mathbb{R}^{3}$, $|\boldsymbol{q}|=1$. Induced by this dipole point source, the
incident electromagnetic fields are
\begin{equation}\label{PSEi}
\boldsymbol{E}^{i}(\boldsymbol{x})=\boldsymbol{G}_{1}(\boldsymbol{x}-\boldsymbol
{x}_s)\boldsymbol{q},\quad \boldsymbol{H}^{i}(\boldsymbol{x}
)=\frac{1}{\mathrm{i}\omega\mu_{1}}(\nabla\times\boldsymbol{E}^{i}(\boldsymbol{x})),\quad\boldsymbol x\in \Omega_1.
\end{equation}
Hence the current source $\boldsymbol{J}_{cs}$ satisfies
\begin{equation*}\label{CS}
\mathrm{i}\omega\mu_{1}\boldsymbol{J}_{cs}(\boldsymbol{x}) =
\boldsymbol{q}\delta(\boldsymbol{x}-\boldsymbol{x}_s),
\quad\boldsymbol x\in\Omega_{1}.
\end{equation*}

Denote by $\mathcal{T}_{j}$ the set of functions $\boldsymbol{\psi}\in
C^{2} (\Omega_{j}) \cap C^{0,\alpha}(\overline{\Omega}_{j}), j=1, 2$.
The direct scattering problem can be stated as follows.

\begin{problem}\label{SP}
Given the incident field $\boldsymbol E^i$ in \eqref{PSEi}, the
direct  problem is to determine $\boldsymbol{E}^{s}\in\mathcal{T}_{1}$
and $\boldsymbol{E}_{2}\in\mathcal{T}_{2}$ such that

\begin{enumerate}[label=(\roman*)]

\item The electric fields $\boldsymbol{E}_{1}=\boldsymbol E^s+\boldsymbol E^i$
and $\boldsymbol{E}_{2}$ satisfy \eqref{Eeq1} and
\eqref{Eeq2}, respectively;

\item The electric field $\boldsymbol{E}_1$ satisfies the boundary condition
\eqref{EBCS1};

\item The electromagnetic fields $(\boldsymbol{E}_{j},
\boldsymbol{H}_{j}), j=1, 2$ satisfy \eqref{EBCS};

\item The scattered fields $(\boldsymbol E^s, \boldsymbol H^s)$ and the
transmitted fields $(\boldsymbol E_2, \boldsymbol H_2)$ satisfy the radiation
conditions \eqref{EsIRC} and \eqref{E2IRC}, respectively.

\end{enumerate}
\end{problem}

It requires to study the dyadic Green function in order to find the integral
representation of the solution for the scattering problem. The details may be
found in \cite{B-07} on the general properties of the dyadic Green function.

\begin{lemma}\label{EL02}
For each fixed $\boldsymbol{y}\in \Omega_{j}$, the dyadic Green function
$\boldsymbol{G}_{j}$ given in \eqref{G} admits the asymptotic
behaviour
\begin{align*}
\boldsymbol{G}_{j}(\boldsymbol{x}-\boldsymbol{y})&=\mathcal O\left(\frac{\exp{
(-\Im(\kappa_{j})|\boldsymbol{x}|)}}{|\boldsymbol{x}|}\right)\boldsymbol{\hat{I}
} \quad \text{as} \ |\boldsymbol{x}-\boldsymbol{y}|\to\infty,\\
\nabla_{\boldsymbol{x}}\times{\boldsymbol{G}_{j}(\boldsymbol{x}-\boldsymbol{y})
}&=\mathcal
O\left(\frac{\exp{(-\Im(\kappa_{j})|\boldsymbol{x}|)}}{|\boldsymbol{x } | }
\right)\boldsymbol{\hat{I}}\quad\text{as}\
|\boldsymbol{x}-\boldsymbol{y}|\to\infty,
\end{align*}
where $\boldsymbol{\hat{I}}:=\boldsymbol{\hat{e}}\boldsymbol{\hat{e}}$ and
$\boldsymbol{\hat{e}}=(1,1,1)^{\top}$.
\end{lemma}

\begin{proof}
Following
\[
|\boldsymbol{x}-\boldsymbol{y}|=\sqrt{|\boldsymbol{x}|^{2}-2\boldsymbol{x}
\cdot\boldsymbol{y}+|\boldsymbol{y}|^{2}}
=|\boldsymbol{x}|-\boldsymbol{\hat{x}}\cdot\boldsymbol{y}+\mathcal
O\left(\frac{1 } { |\boldsymbol{x}|}\right)\quad \text{as}\
|\boldsymbol{x}|\to\infty,
\]
where $\hat{\boldsymbol x}=\boldsymbol x/|\boldsymbol x|$, we have
\begin{align}\label{EL02-2}
\frac{\exp{(\mathrm{i}\kappa_{j}|\boldsymbol{x}-\boldsymbol{y}|)}}{|\boldsymbol{
x}-\boldsymbol{y}|} =&\frac{\exp{(\mathrm{i}\kappa_{j}|\boldsymbol{x}|)}}{
|\boldsymbol{x}|}\notag\\
&\quad \times\bigg\{\exp{(-\mathrm{i}\kappa_{j}\boldsymbol{\hat{x}}
\cdot\boldsymbol{y})} +\mathcal
O\left(\frac{1}{|\boldsymbol{x}|}\right)\bigg\}\quad \text{as}\
|\boldsymbol{x}|\to\infty
\end{align}
uniformly for all $\boldsymbol{y}$ satisfying
$|\boldsymbol{x}-\boldsymbol{y}|\rightarrow\infty$. By \eqref{EL02-2},
for $\Im\kappa_{j}>0$, we obtain for $|\boldsymbol x|\to\infty$ that
\begin{align*}
\boldsymbol{G}_{j}(\boldsymbol{x}-\boldsymbol{y})
&=\bigg[\boldsymbol{I}+\frac{\nabla_{\boldsymbol{y}}\nabla_{\boldsymbol{y}}}{
\kappa_{j}^{2}}\bigg]
\frac{\exp{(\mathrm{i}\kappa_{j}|\boldsymbol{x}|)}}{4\pi
|\boldsymbol{x}|}\bigg\{
\exp{(-\mathrm{i}\kappa_{j}\boldsymbol{\hat{x}}\cdot\boldsymbol{y})}
+\mathcal O\left(\frac{1}{|\boldsymbol{x}|}\right)\bigg\}\\
&=\frac{\exp{(\mathrm{i}\kappa_{j}|\boldsymbol{x}|)}}{4\pi|\boldsymbol{x}|}
\bigg\{[\boldsymbol{I}-\boldsymbol{\hat{x}}\boldsymbol{\hat{x}}]\exp{(-\mathrm{i
}\kappa_{j}\boldsymbol{\hat{x}}\cdot\boldsymbol{y})} +\mathcal
O\left(\frac{1}{|\boldsymbol{x}|}\right)\boldsymbol{\hat{I}}\bigg\}\\
&=\mathcal O\left(\frac{\exp{(-\Im{(\kappa_{j})}|\boldsymbol{x}|)}}{|\boldsymbol{x
}|}\right)\boldsymbol{\hat{I}}
\end{align*}
and
\begin{align*}
\nabla_{\boldsymbol{x}}\times\boldsymbol{G}_{j}(\boldsymbol{x}-\boldsymbol{y})
&=-\nabla_{\boldsymbol{y}}\times\boldsymbol{G}_{j}(\boldsymbol{x}-\boldsymbol{y}
)\\
&=\frac{\exp{(\mathrm{i}\kappa_{j}|\boldsymbol{x}|)}}{4\pi|\boldsymbol{x}|}
\bigg\{-\nabla_{\boldsymbol{y}}\times\bigg[(\boldsymbol{I}-\boldsymbol{\hat{x}}
\boldsymbol{\hat{x}})
\exp{(-\mathrm{i}\kappa_{j}\boldsymbol{\hat{x}}\cdot\boldsymbol{y})}\bigg]
+\mathcal O\left(\frac{1}{|\boldsymbol{x}|}\right)\boldsymbol{\hat{I}}\bigg\}\\
&=\mathrm{i}\kappa_{j}\frac{\exp{(\mathrm{i}\kappa_{j}|\boldsymbol{x}|)}}{
4\pi|\boldsymbol{x}|} \bigg\{\boldsymbol{\hat{x}}\times[(\boldsymbol{I}
-\boldsymbol{\hat{x}}
\boldsymbol{\hat{x}})\exp{(-\mathrm{i}\kappa_{j}\boldsymbol {\hat{x}}
\cdot\boldsymbol{y})}]
+\mathcal O\left(\frac{1}{|\boldsymbol{x}|}\right)\boldsymbol{\hat{I}}\bigg\}\\
&=\mathcal
O\left(\frac{\exp{(-\Im{(\kappa_{j})}|\boldsymbol{x}|)}}{|\boldsymbol{x } | }
\right)\boldsymbol{\hat{I}},
\end{align*}
which completes the proof.
\end{proof}

We introduce some Banach spaces. For
$V \subset \mathbb{R}^3$, denote by $BC(V)$ the set of bounded and
continuous functions on $V$, which is a Banach space under the norm
\[
\|\phi\|_{\infty}=\sup\limits_{x\in V}|\phi(x)|.
\]
For $0<\alpha\leq 1$, denote by $C^{0,\alpha}(V)$ the Banach
space of functions $\phi\in BC(V)$ which are uniformly
H\"{o}lder continuous with exponent $\alpha$. The norm
$\|\cdot\|_{C^{0,\alpha}(V)}$ is defined by
\[
\|\phi\|_{C^{0,\alpha}(V)}=\|\phi\|_{\infty}+\sup\limits_{\substack{x,y\in
V\\ x\neq y}}\frac{|\phi(x)-\phi(y)|}{|x-y|^{\alpha}}.
\]
Let $
C^{1,\alpha}(V)=\{\phi\in BC(V)\cap C^{1}(V) : \nabla\phi\in C^{0,\alpha}(V)\},
$
which is a Banach space under the norm
\[
\|\phi\|_{C^{1,\alpha}(V)}=\|\phi\|_{\infty}+\|\nabla\phi\|_{C^{0,\alpha}(V)}.
\]

\section{Well-posedness of the direct problem}

In this section, we show the existence and uniqueness of the solution to
Problem \ref{SP} by using the boundary integral equation method. First we derive
an integral representation for the solution of Problem \ref{SP} using dyadic
Green's theorem combined with the radiation conditions \eqref{EsIRC} and
\eqref{E2IRC}.

\begin{theorem}\label{ET1}
Let the fields $(\boldsymbol{E}_{1}, \boldsymbol{E}_{2})$ be the solution of
Problem \ref{SP}, then $(\boldsymbol{E}_{1}, \boldsymbol{E}_{2})$ have the
integral representations
\begin{align}\label{E1}
\boldsymbol{E}_{1}(\boldsymbol{x})
=\boldsymbol{E}^{i}(\boldsymbol{x}) + & \int_{S}
\big\{[\mathrm{i}\omega\mu_{1}\boldsymbol{G}_{1}(\boldsymbol{x}-\boldsymbol{y})]
\cdot[\boldsymbol{\nu}_{S}(\boldsymbol{y})\times\boldsymbol{H}_{1}(\boldsymbol{y
})]+\notag\\
&\qquad [\nabla_{\boldsymbol{x}}\times\boldsymbol{G}_{1}(\boldsymbol{x}
-\boldsymbol
{y})]\cdot[\boldsymbol{\nu}_{S}(\boldsymbol{y})\times\boldsymbol{E}_ {1}
(\boldsymbol{y })]\big\}{\rm d}s_{\boldsymbol{y}}\nonumber\\
 + & \int_{\Gamma} \big\{[\mathrm{i}\omega\mu_{1}\boldsymbol{G}_{1}(\boldsymbol
{x}-\boldsymbol{y})] \cdot[\boldsymbol{\nu}_{\Gamma}(\boldsymbol{y}
)\times\boldsymbol{H}_{1} (\boldsymbol{y})]\big\}{\rm d}s_{\boldsymbol{y}},
\quad\boldsymbol{x}\in \Omega_{1},
\end{align}
and
\begin{align}\label{E2}
\boldsymbol{E}_{2}(\boldsymbol{x})
= & -\int_{S}
\big\{[\mathrm{i}\omega\mu_{2}\boldsymbol{G}_{2}(\boldsymbol{x}-\boldsymbol{y})]
\cdot[\boldsymbol{\nu}_{S}(\boldsymbol{y})\times\boldsymbol{H}_{2}(\boldsymbol{y
})]\nonumber\\
&\qquad+[\nabla_{\boldsymbol{x}}\times\boldsymbol{G}_{2}(\boldsymbol{x}
-\boldsymbol{y})]
\cdot[\boldsymbol{\nu}_{\Gamma}(\boldsymbol{y})\times\boldsymbol{E}_{2}
(\boldsymbol{y})]\big\}{\rm d}s_{\boldsymbol{y}},\quad\boldsymbol{x}\in
\Omega_{2}.
\end{align}
\end{theorem}

\begin{proof}
Let $B_r=\{\boldsymbol x\in\mathbb R^3: |\boldsymbol x|<r\}$. Denote
$\Omega_{r}=B_{r}\cap \Omega_{1}$ with the boundary $\partial
\Omega_{r}=\partial B_{r}^{+}\cup \Gamma\cup S_{r}$,
where $\partial B_{r}^{+}=\partial B_{r}\cap \Omega_{1}$ and $S_{r}=S\cap
B_{r}$. For each fixed $\boldsymbol{x}\in \Omega_{r}$,
applying the vector dyadic Green second theorem to $\boldsymbol{E}_{1}$ and
$\boldsymbol{G}_{1}$  in the region $\Omega_{r}$, we obtain
\begin{align}\label{ET1-1}
&\int_{\Omega_{r}}\{\boldsymbol{E}_{1}(\boldsymbol{y})
\cdot[\nabla_{\boldsymbol{y}}\times\nabla_{\boldsymbol{y}}\times\boldsymbol{G}_{
1}(\boldsymbol{y}-\boldsymbol{x})]
-[\nabla_{\boldsymbol{y}}\times\nabla_{\boldsymbol{y}}\times\boldsymbol{E}_{1}
(\boldsymbol{y})]\cdot\boldsymbol{G}_{1}(\boldsymbol{y}-\boldsymbol{x})\}
{\rm d}{\boldsymbol{y}}\nonumber\\
&=-\int_{\partial \Omega_{r}}
\{[\boldsymbol{\nu}(\boldsymbol{y})\times(\nabla_{\boldsymbol{y}}
\times\boldsymbol{E}_{1}(\boldsymbol{y}))]
\cdot\boldsymbol{G}_{1}(\boldsymbol{y}-\boldsymbol{x})\notag\\
&\hspace{1.7cm}
+[\boldsymbol{\nu}(\boldsymbol{y})\times\boldsymbol{E}_{1}(\boldsymbol{y} ) ]
\cdot[\nabla_{\boldsymbol{y}}\times\boldsymbol{G}_{1}(\boldsymbol{y}-\boldsymbol
{x})]\}{\rm d}s_{\boldsymbol{y}},
\end{align}
where $\boldsymbol{\nu}=\boldsymbol{\nu}(\boldsymbol{y})$ stands for the unit
normal vector at
$\boldsymbol{y}\in\partial \Omega_{r}$ pointing out of $\Omega_{r}$.

It follows from \eqref{Eeq1} and \eqref{G1eq} that
\begin{align*}
&\int_{\Omega_{r}}\{\boldsymbol{E}_{1}(\boldsymbol{y})
\cdot[\nabla_{\boldsymbol{y}}\times\nabla_{\boldsymbol{y}}\times\boldsymbol{G}_{
1}(\boldsymbol{y}-\boldsymbol{x})]
-[\nabla_{\boldsymbol{y}}\times\nabla_{\boldsymbol{y}}\times\boldsymbol{E}_{1}
(\boldsymbol{y})]\cdot\boldsymbol{G}_{1}(\boldsymbol{y}-\boldsymbol{x})\}
{\rm d}{\boldsymbol{y}}\\
&=\int_{\Omega_{r}}[\boldsymbol{E}_{1}(\boldsymbol{y})]\cdot[\nabla_{\boldsymbol
{y}}\times\nabla_{\boldsymbol{y}}\times\boldsymbol{G}_{1}(\boldsymbol{y}
-\boldsymbol{x})
-\kappa_{1}^{2}\boldsymbol{G}_{1}(\boldsymbol{y}-\boldsymbol{x})]{\rm d}{
\boldsymbol{y}}\\
&\qquad -\int_{\Omega_{r}}[\nabla_{\boldsymbol{y}}\times\nabla_{\boldsymbol{y}}
\times\boldsymbol{E}_{1}(\boldsymbol{y})
-\kappa_{1}^{2}\boldsymbol{E}_{1}(\boldsymbol{y})]\cdot[\boldsymbol{G}_{1}
(\boldsymbol{y}-\boldsymbol{x})]{\rm d}{\boldsymbol{y}}\nonumber\\
&=\int_{\Omega_{r}}[\boldsymbol{E}_{1}(\boldsymbol{y})\cdot(\delta(\boldsymbol{y
}-\boldsymbol{x})\boldsymbol{I} )]{\rm d}{\boldsymbol{y}}
-\int_{\Omega_{r}}[\mathrm{i}\omega\mu_{1}\boldsymbol{J}_{cs}(\boldsymbol{y}
)\cdot\boldsymbol{G}_{1}(\boldsymbol{y}-\boldsymbol{x})]{\rm
d}{\boldsymbol{y}}\\
&=\boldsymbol{E}_{1}(\boldsymbol{x})
-\int_{\Omega_{r}}[\mathrm{i}\omega\mu_{1}\boldsymbol{J}_{cs}(\boldsymbol{y}
)\cdot\boldsymbol{G}_{1}(\boldsymbol{y}-\boldsymbol{x})]{\rm
d}{\boldsymbol{y}},
\end{align*}
where
\begin{align}\label{ET1-3}
\lim_{r\to +\infty}\int_{\Omega_{r}}[\mathrm{i}\omega\mu_{1}\boldsymbol
{J}_{cs}(\boldsymbol{y})\cdot\boldsymbol{G}_{1}(\boldsymbol{y}-\boldsymbol{x})]{
\rm d} {\boldsymbol{y}}& =\int_{\Omega_{1}}[\boldsymbol{q}\delta(\boldsymbol{y}
-\boldsymbol{x}_s)\cdot\boldsymbol{G}_{1}(\boldsymbol{y}
-\boldsymbol{x})]{\rm d}{\boldsymbol{y}}\nonumber\\
&=\boldsymbol{G}_{1}(\boldsymbol{x}-\boldsymbol{x}_s)\boldsymbol{
q}=\boldsymbol{E}^{i}(\boldsymbol{x}).
\end{align}
Hence, letting $r\to +\infty$, with the aid of
\eqref{ET1-1}--\eqref{ET1-3}, we have
\begin{align}\label{ET1-4}
\boldsymbol{E}_{1}(\boldsymbol{x})-\boldsymbol{E}^{i}(\boldsymbol{x}
)&=-\int_{\partial \Omega_{1}}
\big\{[\boldsymbol{\nu}(\boldsymbol{y})\times(\nabla_{\boldsymbol{y}}
\times\boldsymbol{E}_{1}(\boldsymbol{y}))]
\cdot\boldsymbol{G}_{1}(\boldsymbol{y}-\boldsymbol{x})\notag\\
&\hspace{2cm} +[\boldsymbol{\nu}
(\boldsymbol{y})\times\boldsymbol{E}_{1}(\boldsymbol{y})]
\cdot[\nabla_{\boldsymbol{y}}\times\boldsymbol{G}_{1}(\boldsymbol{y}-\boldsymbol
{x})]\big\}{\rm d}s_{\boldsymbol{y}}\nonumber\\
&=-\left(\int_{S}+\int_{\Gamma}+\lim_{r\to +\infty}\int_{\partial
B_{r}^{+}}\right)
\big\{[\boldsymbol{\nu}(\boldsymbol{y})\times(\nabla_{\boldsymbol{y}}
\times\boldsymbol{E}_{1}(\boldsymbol{y}))]
\cdot\boldsymbol{G}_{1}(\boldsymbol{y}-\boldsymbol{x})\nonumber\\
&\hspace{2cm}+[\boldsymbol{
\nu}(\boldsymbol{y})\times\boldsymbol{E}_{1}(\boldsymbol{y})]
\cdot[\nabla_{\boldsymbol{y}}\times\boldsymbol{G}_{1}(\boldsymbol{y}-\boldsymbol
{x})]\big\}{\rm d}s_{\boldsymbol{y}}.
\end{align}
Following Lemma \ref{EL02} and \eqref{EsIRC}--\eqref{E2IRC}, we obtain for $r\to
+\infty$ that
\begin{align}\label{ET1-5}
&\bigg|\int_{\partial B_{r}^{+}}
\big\{[\boldsymbol{\nu}(\boldsymbol{y})\times(\nabla_{\boldsymbol{y}}
\times\boldsymbol{E}^{s}(\boldsymbol{y}))]
\cdot\boldsymbol{G}_{1}(\boldsymbol{y} -\boldsymbol{x})\notag\\
&\hspace{2cm} +[\boldsymbol{\nu}(\boldsymbol{y})\times\boldsymbol{E}^{s}
(\boldsymbol{y})] \cdot[\nabla_{\boldsymbol{y}}\times\boldsymbol{G}_{1}
(\boldsymbol{y}-\boldsymbol {x})]\big\}{\rm
d}s_{\boldsymbol{y}}\bigg|\notag\\
&\leq\bigg[\omega^{2}\mu_{1}^{2}\int_{\partial B_{r}^{+}}
|\boldsymbol{H}^{s}(\boldsymbol{y})|^{2}{\rm
d}s_{\boldsymbol{y}} \bigg]^{\frac{1}{2}}
\cdot\bigg[\int_{\partial B_{r}^{+}}
|\boldsymbol{G}_{1}(\boldsymbol{y}-\boldsymbol{x})|^{2}{\rm d}s_{\boldsymbol{y}}
\bigg ]^{\frac{1}{2}}\notag\\
&\hspace{2cm} +\bigg[\int_{\partial B_{r}^{+}}
|\boldsymbol{E}^{s}(\boldsymbol{y})|^{2}{\rm
d}s_{\boldsymbol{y}}\bigg]^{\frac{1}{ 2 } }
\cdot\bigg[\int_{\partial B_{r}^{+}}
|\nabla_{\boldsymbol{y}}\times\boldsymbol{G}_{1}(\boldsymbol{y}-\boldsymbol{x}
)|^{2}{ \rm d}s_{\boldsymbol{y}}\bigg]^{\frac{1}{2}}\to 0.
\end{align}
By Lemma \ref{EL02} and the definition of incident field $\boldsymbol E^i$, we
have for $r\to +\infty$ that
\begin{align}\label{ET1-6}
&\bigg|\int_{\partial B_{r}^{+}}
\big\{[\boldsymbol{\nu}(\boldsymbol{y})\times(\nabla_{\boldsymbol{y}}
\times\boldsymbol{E}^{i}(\boldsymbol{y}))]
\cdot\boldsymbol{G}_{1}(\boldsymbol{y}-\boldsymbol{x})\notag\\
&\hspace{2cm}+[\boldsymbol{\nu}
(\boldsymbol{y})\times\boldsymbol{E}^{i}(\boldsymbol{y})]
\cdot[\nabla_{\boldsymbol{y}}\times\boldsymbol{G}_{1}(\boldsymbol{y}-\boldsymbol
{x})]\big\}{\rm d}s_{\boldsymbol{y}}\bigg|\notag\\
&\leq\bigg[\omega^{2}\mu_{1}^{2}\int_{\partial B_{r}^{+}}
|\boldsymbol{H}^{i}(\boldsymbol{y})|^{2}{\rm
d}s_{\boldsymbol{y}} \bigg]^{\frac{1}{2}}
\cdot\bigg[\int_{\partial B_{r}^{+}}
|\boldsymbol{G}_{1}(\boldsymbol{y}-\boldsymbol{x})|^{2}{\rm d}s_{\boldsymbol{y}}
\bigg ]^{\frac{1}{2}}\notag\\
&\hspace{2cm} +\bigg[\int_{\partial B_{r}^{+}}
|\boldsymbol{E}^{i}(\boldsymbol{y})|^{2}{\rm
d}s_{\boldsymbol{y}}\bigg]^{\frac{1}{ 2 } }
\cdot\bigg[\int_{\partial B_{r}^{+}}
|\nabla_{\boldsymbol{y}}\times\boldsymbol{G}_{1}(\boldsymbol{y}-\boldsymbol{x}
)|^{2}{ \rm d}s_{\boldsymbol{y}}\bigg]^{\frac{1}{2}}\to 0.
\end{align}

Using \eqref{ET1-4}--\eqref{ET1-6} and conditions (ii), (iv) in Problem
\ref{SP}, and letting $r\to +\infty$, we have for each fixed $\boldsymbol{x}\in
\Omega_{1}$ that
\begin{align*}
\boldsymbol{E}_{1}(\boldsymbol{x})-\boldsymbol{E}^{i}(\boldsymbol{x})
=& -\int_{S}
\big\{[\boldsymbol{\nu}(\boldsymbol{y})\times(\nabla_{\boldsymbol{y}}
\times\boldsymbol{E}_{1}(\boldsymbol{y}))]
\cdot\boldsymbol{G}_{1}(\boldsymbol{y}
-\boldsymbol{x})\\
&\qquad +[\boldsymbol{\nu}(\boldsymbol{y})\times\boldsymbol{E}_{1}
(\boldsymbol{y})] \cdot[\nabla_{\boldsymbol{y}}\times\boldsymbol{G}_{1}
(\boldsymbol{y}-\boldsymbol {x})]\big\}{\rm d}s_{\boldsymbol{y}}\\
& -\int_{\Gamma}
\big\{[\boldsymbol{\nu}(\boldsymbol{y})\times(\nabla_{\boldsymbol{ y}}
\times\boldsymbol{E}_{1}(\boldsymbol{y}))]
\cdot\boldsymbol{G}_{1}(\boldsymbol{y} -\boldsymbol{x})\big\}{\rm
d}s_{\boldsymbol{y}}\\
=&\int_{S}\big\{[\mathrm{i}\omega\mu_{1}\boldsymbol{G}_{1}(\boldsymbol{x}
-\boldsymbol{y})]\cdot[\boldsymbol{\nu}_{S}(\boldsymbol{y})\times\boldsymbol{H}_
{1}(\boldsymbol{y})]\\
&\qquad +[\nabla_{\boldsymbol{x}}\times\boldsymbol{G}_{1}
(\boldsymbol{x}-\boldsymbol{y})]
\cdot[\boldsymbol{\nu}_{S}(\boldsymbol{y})\times\boldsymbol{E}_{1}(\boldsymbol{y
})]\big\}{\rm d}s_{\boldsymbol{y}}\\
&+\int_{\Gamma}\big\{[\mathrm{i}\omega\mu_{1}\boldsymbol{G}_{1}(\boldsymbol {x}
-\boldsymbol{y})] \cdot[\boldsymbol{\nu}_{\Gamma}(\boldsymbol{y}
)\times\boldsymbol{H}_{1}
(\boldsymbol{y})]\big\}{\rm d}s_{\boldsymbol{y}}.
\end{align*}
Similarly, for each fixed $\boldsymbol{x}\in \Omega_{2}$, we have
\begin{align*}
\boldsymbol{E}_{2}(\boldsymbol{x})
=&-\int_{S} \big\{[\boldsymbol{G}_{2}(\boldsymbol{x}-\boldsymbol{y})]
\cdot[\boldsymbol{\nu}_{S}
(\boldsymbol{y})\times(\nabla_{\boldsymbol{y}}\times\boldsymbol{E}_{2}
(\boldsymbol{y}))]\\
&\qquad+[\nabla_{\boldsymbol{x}}\times\boldsymbol{G}_{2}(\boldsymbol{
x}-\boldsymbol{y})]
\cdot[\boldsymbol{\nu}_{S}(\boldsymbol{y})\times\boldsymbol{E
}_{2}(\boldsymbol{y })]\big\}{\rm d}s_{\boldsymbol{y}}\\
=& -\int_{S}
\big\{[\mathrm{i}\omega\mu_{2}\boldsymbol{G}_{2}(\boldsymbol{x}-\boldsymbol{y})]
\cdot[\boldsymbol{\nu}_{S}(\boldsymbol{y})\times\boldsymbol{H}_{2}(\boldsymbol{y
})]\\
&\qquad+[\nabla_{\boldsymbol{x}}\times\boldsymbol{G}_{2}(\boldsymbol{x}
-\boldsymbol{y})]
\cdot[\boldsymbol{\nu}_{S}(\boldsymbol{y})\times\boldsymbol{E}_{2}(\boldsymbol{y
})]\big\}{\rm d}s_{\boldsymbol{y}},
\end{align*}
where
\begin{align*}
\boldsymbol{\nu}_{S}(\boldsymbol{y})\times\boldsymbol{E}_j(\boldsymbol{y})
&=\lim_{h\to +0}\boldsymbol{\nu}_{S}(\boldsymbol{y})\times\boldsymbol{E}
_j(\boldsymbol{y}+(-1)^{j}h\boldsymbol{\nu}_{S}(\boldsymbol{y})),\\
\boldsymbol{\nu}_{S}(\boldsymbol{y})\times[\nabla_{\boldsymbol{y}}
\times\boldsymbol{E}_j(\boldsymbol{y})]
&=\lim_{h\to
+0}\boldsymbol{\nu}_{S}(\boldsymbol{y})\times[\nabla_{\boldsymbol{y}}
\times\boldsymbol{E}_j(\boldsymbol{y}+(-1)^{j}h\boldsymbol{\nu}_{S}(\boldsymbol{y}))]
\end{align*}
are to be understood in the sense of uniform convergence on $S$, and $j=1,2$.

Finally, from the jump relations and \eqref{EBCS1}, we note that the integral
representations \eqref{E1}--\eqref{E2} lead to the boundary integral
equations:
\begin{align}\label{E1S}
\frac{1}{2}\boldsymbol{\nu}_{S}(\boldsymbol{x})\times\boldsymbol{E}_{1}
(\boldsymbol{x})=&\boldsymbol{\nu}_{S}(\boldsymbol{x})\times\boldsymbol{E}^{i}
(\boldsymbol{x})+\int_{S}
\big\{[\mathrm{i}\omega\mu_{1}\boldsymbol{\nu}_{S}(\boldsymbol{x}
)\times\boldsymbol {G}_{1}(\boldsymbol{x}-\boldsymbol{y})]
\cdot[\boldsymbol{\nu}_{S}(\boldsymbol{y})\times\boldsymbol{H}_{1}(\boldsymbol{y
})]\nonumber\\
&+[\boldsymbol{\nu}_{S}(\boldsymbol{x})\times(\nabla_{\boldsymbol{x}}
\times\boldsymbol{G}_{1}(\boldsymbol{x}-\boldsymbol{y}))]
\cdot[\boldsymbol{{\nu}}_{S}(\boldsymbol{y})\times\boldsymbol{E}_{1}(\boldsymbol
{y})]\big\}{\rm d}s_{\boldsymbol{y}}\nonumber\\
&+\int_{\Gamma}
\big\{[\mathrm{i}\omega\mu_{1}\boldsymbol{\nu}_{S}(\boldsymbol{x}
)\times\boldsymbol {G}_{1}(\boldsymbol{x}-\boldsymbol{y})]
\cdot[\boldsymbol{\nu}_{\Gamma}(\boldsymbol{y})\times\boldsymbol{H}_{1}
(\boldsymbol{y})]\big\}{\rm d}s_{\boldsymbol{y}},\quad \boldsymbol{x}\in S,
\end{align}
\begin{align}\label{E1Gam}
0=&\boldsymbol{\nu}_{\Gamma}(\boldsymbol{x})\times\boldsymbol{E}^{i}(\boldsymbol
{x})+\int_{S}
\big\{[\mathrm{i}\omega\mu_{1}\boldsymbol{\nu}_{\Gamma}(\boldsymbol{x}
)\times\boldsymbol{G}_{1}(\boldsymbol{x}-\boldsymbol{y})]
\cdot[\boldsymbol{\nu}_{S}(\boldsymbol{y})\times\boldsymbol{H}_{1}(\boldsymbol{y
})]\nonumber\\
&+[\boldsymbol{\nu}_{\Gamma}(\boldsymbol{x})\times(\nabla_{\boldsymbol
{x}}\times\boldsymbol{G}_{1}(\boldsymbol{x}-\boldsymbol{y}))]
\cdot[\boldsymbol{{\nu}}_{S}(\boldsymbol{y})\times\boldsymbol{E}_{1}(\boldsymbol
{y})]\big\}{\rm d}s_{\boldsymbol{y}}\nonumber\\
&+\int_{\Gamma}
\big\{[\mathrm{i}\omega\mu_{1}\boldsymbol{\nu}_{\Gamma}(\boldsymbol{x}
)\times\boldsymbol{G}_{1}(\boldsymbol{x}-\boldsymbol{y})]
\cdot[\boldsymbol{\nu}_{\Gamma}(\boldsymbol{y})\times\boldsymbol{H}_{1}
(\boldsymbol{y})]\big\}{\rm d}s_{\boldsymbol{y}},\quad \boldsymbol{x}\in
\Gamma,
\end{align}
and
\begin{align}\label{E2S}
\frac{1}{2}\boldsymbol{\nu}_{S}(\boldsymbol{x})\times\boldsymbol{E}_{2}
(\boldsymbol{x})=& -\int_{S}
\big\{[\mathrm{i}\omega\mu_{2}\boldsymbol{\nu}_{S}(\boldsymbol{x}
)\times\boldsymbol {G}_{2}(\boldsymbol{x}-\boldsymbol{y})]
\cdot[\boldsymbol{\nu}_{S}(\boldsymbol{y})\times\boldsymbol{H}_{2}(\boldsymbol{y
})]\nonumber\\
& +[\boldsymbol{\nu}_{S}(\boldsymbol{x})\times(\nabla_{\boldsymbol{x}}
\times\boldsymbol{G}_{2}(\boldsymbol{x}-\boldsymbol{y}))]
\cdot[\boldsymbol{\nu}_{S}(\boldsymbol{y})\times\boldsymbol{E}_{2}(\boldsymbol{y
})]\big\}{\rm d}s_{\boldsymbol{y}},\quad \boldsymbol{x}\in S.\quad
\end{align}
Hence, the electric fields $(\boldsymbol E_1, \boldsymbol E_2)$ satisfy the
boundary integral equations \eqref{E1S}--\eqref{E2S} and the continuity
conditions
\begin{equation}\label{E1bcs}
\boldsymbol{\nu}_{S}\times{\boldsymbol{E}_{1}}=\boldsymbol{\nu}_{S}\times{
\boldsymbol{E}_{2}}, \quad
\boldsymbol{\nu}_{S}\times\boldsymbol{H}_{1}=\boldsymbol{\nu}_{S}
\times\boldsymbol{H}_{2}\quad \text{on}\ S,
\end{equation}
which completes the proof.
\end{proof}

To show the well-posedness of the boundary integral equations
\eqref{E1S}--\eqref{E2S}, we introduce the normed subspace of continuous
tangential fields
\[
\mathfrak{T}(S):=\{\boldsymbol{\psi}\in
C(S): \boldsymbol{\nu}_{S}\cdot\boldsymbol{\psi}=0\},
\]
and the normed space of uniformly H\"{o}lder continuous tangential fields
\[
\mathfrak{T}^{0,\alpha}(S):=\{\boldsymbol{\psi}\in\mathfrak{T}(S)|\
\boldsymbol{\psi}\in C^{0,\alpha}(S)\}.
\]

We consider the integral operator $T :
\mathfrak{T}^{0,\alpha}(S) \to \mathfrak{T}^{0,\alpha}(S)$ defined by
\begin{align}\label{CP1}
(T\boldsymbol{\Psi})(\boldsymbol{x})&=\int_{S}
[\mathrm{i}\omega\mu_{1}\boldsymbol{\nu}_{S}(\boldsymbol{x})\times\boldsymbol{G}
_{1}(\boldsymbol{x}-\boldsymbol{y})]
\cdot[\boldsymbol{\Psi}(\boldsymbol{y})]
{\rm d}s_{\boldsymbol{y}}\nonumber\\
&=\int_{\mathbb R^2}
[\mathrm{i}\omega\mu_{1}(\boldsymbol{\nu}_{S}(\boldsymbol{x})\times\boldsymbol{G
}_{1}(\boldsymbol{x}-\boldsymbol{y}))
\cdot\boldsymbol{\Psi}(\boldsymbol{y})|_{y_{3}=f(y_{1}, y_{2})}]
(1+f_{y_{1}}^{2}+f_{y_{2}}^{2})^{1/2}
{\rm d}y_{1}{\rm d}y_{2},
\end{align}
and the integral operator $K : \mathfrak{T}^{0,\alpha}(S) \to
\mathfrak{T}^{0,\alpha}(S)$ defined by
\begin{align}\label{CP2}
(K\boldsymbol{\Phi})(\boldsymbol{x})&=\int_{S}[\boldsymbol{\nu}_{S}(\boldsymbol{
x})\times(\nabla_{\boldsymbol{x}}
\times\boldsymbol{G}_{1}(\boldsymbol{x}-\boldsymbol{y}))]
\cdot[\boldsymbol{\Phi}(\boldsymbol{y})]{\rm d}s_{\boldsymbol{y}}\nonumber\\
&=\int_{\mathbb R^2}
[\mathrm{i}\omega\mu_{1}(\boldsymbol{\nu}_{S}(\boldsymbol{x})\times\boldsymbol{G
}_{1}(\boldsymbol{x}-\boldsymbol{y}))
\cdot\boldsymbol{\Psi}(\boldsymbol{y})|_{y_{3}=f(y_{1}, y_{2})}]
(1+f_{y_{1}}^{2}+f_{y_{2}}^{2})^{1/2} {\rm d}y_{1}{\rm d}y_{2}.
\end{align}
For each $n\in\mathbb Z^+$, define the truncated operator $T_{n}:
\mathfrak{T}^{0,\alpha}(S_{n}) \to
\mathfrak{T}^{0,\alpha}(S_{n})$ by
\begin{align}\label{CPn1}
(T_{n}\boldsymbol{\Psi})(\boldsymbol{x})
=\int_{-n}^{n}\int_{-n}^{n}
[\mathrm{i}\omega\mu_{1}(\boldsymbol{\nu}_{S}(\boldsymbol{x})\times\boldsymbol{G
}_{1}(\boldsymbol{x}-\boldsymbol{y}))
\cdot\boldsymbol{\Psi}(\boldsymbol{y})|_{y_{3}=f(y_{1}, y_{2})}]\notag\\
\times(1+f_{y_{1}}^{2}+f_{y_{2}}^{2})^{1/2} {\rm d}y_{1}{\rm d}y_{2},
\end{align}
and the operator $K_{n}: \mathfrak{T}^{0,\alpha}(S_{n}) \to
\mathfrak{T}^{0,\alpha}(S_{n})$ by
\begin{align}\label{CPn2}
(K_{n}\boldsymbol{\Phi})(\boldsymbol{x}) =\int_{-n}^{n}\int_{-n}^{n}
[\mathrm{i}\omega\mu_{1}(\boldsymbol{\nu}_{S}(\boldsymbol{x})\times\boldsymbol{G
}_{1}(\boldsymbol{x}-\boldsymbol{y}))
\cdot\boldsymbol{\Psi}(\boldsymbol{y})|_{y_{3}=f(y_{1}, y_{2})}]\notag\\
\times(1+f_{y_{1}}^{2}+f_{y_{2}}^{2})^{1/2} {\rm d}y_{1}{\rm d}y_{2},
\end{align}
where $S_{n}=\{\boldsymbol{x}\in S : |x_{j}|\leq n,\ j=1,2\}$.

It follows from \cite[Theorems 2.32 and 2.33]{CK-83} that the integral
operators $T_{n}$ and $K_{n}$ are compact. We show that the
integral operators $T$ and $K$ are also compact. Hence the boundary integral
equations \eqref{E1S}--\eqref{E2S} are of the Fredholm type, i.e., the
existence of the solution follows immediately from the uniqueness of the
solution.

\begin{lemma}\label{EL01}
The integral operators $T$ and $K$  are compact.
\end{lemma}

\begin{proof}
For each fixed $\boldsymbol{x}\in S_{n}$, it follows from \eqref{CP1} and
\eqref{CPn1} that
\begin{align}\label{EL01-1}
&(T\boldsymbol{\Psi})(\boldsymbol{x})
-(T_{n}\boldsymbol{\Psi})(\boldsymbol{x})\nonumber\\
&=\left(\int_{n}^{+\infty}\int_{-\infty}^{+\infty}+\int_{-\infty}^{-n}\int_{
-\infty}^{+\infty}
+\int_{-n}^{+n}\int_{-\infty}^{-n}+\int_{-n}^{+n}\int_{n}^{+\infty}
\right)\boldsymbol{\varphi}(\boldsymbol{x},y_{1},y_{2})
{\rm d}y_{1}{\rm d}y_{2}\nonumber\\
&=I_{1}+I_{2}+I_{3}+I_{4},
\end{align}
where
\begin{eqnarray*}
\boldsymbol{\varphi}(\boldsymbol{x},y_{1},y_{2}) =[\mathrm{i}\omega\mu_{1}
(\boldsymbol{\nu}_{S}(\boldsymbol{x})\times\boldsymbol{G}_{1}(\boldsymbol{x}
-\boldsymbol{y})) \cdot\boldsymbol{\Psi}(\boldsymbol{y})|_{y_{3}=f(y_{1},
y_{2})}] (1+f_{y_{1}}^{2}+f_{y_{2}}^{2})^{1/2}.
\end{eqnarray*}

By Lemma \ref{EL02}, we have for $n\rightarrow+\infty$ that
\begin{align}\label{EL01-2}
|I_{1}|\nonumber &\leq\int_{n}^{+\infty}\int_{-\infty}^{+\infty}
|\boldsymbol{\varphi}(\boldsymbol{x},y_{1},y_{2})|
{\rm d}y_{1}{\rm d}y_{2}\nonumber\\
&\leq C\int_{n}^{+\infty}\int_{-\infty}^{+\infty}
[|\boldsymbol{G}_{1}(\boldsymbol{x}-\boldsymbol{y})|
\cdot|\boldsymbol{\Psi}(\boldsymbol{y})|_{y_{3}=f(y_{1}, y_{2})}]
{\rm d}y_{1}{\rm d}y_{2}\nonumber\\
&\leq C\|\boldsymbol{\Psi}\|_{C^{0,\alpha}(S)}\int_{n}^{+\infty}\int_{-\infty}^{
+\infty} \bigg(\frac{\exp{(-\Im(\kappa_{1})|\boldsymbol{y}|)}}{|\boldsymbol{y}|}
\bigg|_{y_{3}=f(y_{1}, y_{2})}\bigg){\rm d}y_{1}{\rm d}y_{2}\nonumber\\
&\leq  C\|\boldsymbol{\Psi}\|_{C^{0,\alpha}(S)}
\int_{0}^{+\infty}\exp{\big(-\frac{1}{2}\Im(\kappa_{1})y_{1}\big)}{\rm
d}y_{1} \int_{n}^{+\infty}\frac{\exp{(-\frac{1}{2}\Im(\kappa_{
1})y_{2})}}{y_{2}}{\rm d}y_{2}\nonumber\\
&= C\|\boldsymbol{\Psi}\|_{C^{0,\alpha}(S)}
\left(\frac{2}{\Im(\kappa_{1})}\right)^{2}
\left(\frac{1}{n}\exp{\big(-\frac{n}{2}\Im(\kappa_{1})\big)}
\right)\rightarrow 0,
\end{align}
where $C$ is a positive constant and may change from step to step. Similarly,
we may show for $j=2,3,4$ that
\begin{eqnarray}\label{EL01-3}
|I_{j}|
\leq C\|\boldsymbol{\Psi}\|_{C^{0,\alpha}(S)}
\left(\frac{1}{n}\exp{\big(-\frac{n}{2}\Im(\kappa_{1})\big)}
\right)\rightarrow 0
\quad\text{as}\ n\rightarrow+\infty.
\end{eqnarray}
Combining \eqref{EL01-1}--\eqref{EL01-3} leads to
\begin{align*}
&|(T\boldsymbol{\Psi})(\boldsymbol{x})
-(T_{n}\boldsymbol{\Psi})(\boldsymbol{x})|
\leq \sum_{j=1}^{4}|I_{j}|\\
&\leq C\|\boldsymbol{\Psi}\|_{C^{0,\alpha}(S)}
\left(\frac{1}{n}\exp{\big(-\frac{n}{2}\Im(\kappa_{1})\big)}
\right)\rightarrow 0
\quad\text{as}\ n\rightarrow+\infty.
\end{align*}
Hence we have
\begin{eqnarray}\label{EL01-5}
\|(T-T_{n})\boldsymbol{\Psi}\|_{\infty}\leq
C\|\boldsymbol{\Psi}\|_{C^{0,\alpha}(S)}
\left(\frac{1}{n}\exp{\big(-\frac{n}{2}\Im{(\kappa_{1})}\big)}
\right)\rightarrow 0
\quad\text{as}\ n\rightarrow+\infty.
\end{eqnarray}

For each fixed $\boldsymbol{x}, \boldsymbol{\widetilde{x}}\in S_{n}$ and
$\boldsymbol{x}\neq\boldsymbol{\widetilde{x}}$,
it follows from \eqref{CP1} and \eqref{CPn1} that
\begin{align}\label{EL01-6}
&((T-T_{n})\boldsymbol{\Psi})(\boldsymbol{x})-((T-T_{n})\boldsymbol{\Psi}
)(\boldsymbol{\widetilde{x}})\nonumber\\
&=\left(\int_{n}^{+\infty}\int_{-\infty}^{+\infty}+\int_{-\infty}^{-n}\int_{
-\infty}^{+\infty}
+\int_{-n}^{+n}\int_{-\infty}^{-n}+\int_{-n}^{+n}\int_{n}^{+\infty}
\right)\notag\\
&\hspace{2cm}[\boldsymbol{\varphi}(\boldsymbol{x},y_{1},y_{2})-\boldsymbol{
\varphi }
(\boldsymbol{\widetilde{x}},y_{1},y_{2})]
{\rm d}y_{1}{\rm d}y_{2}\nonumber\\
&=I_{5}+I_{6}+I_{7}+I_{8}.
\end{align}
From Lemma \ref{EL02} and the mean value theorem, we get
\begin{eqnarray*}
|\boldsymbol{G}_{1}(\boldsymbol{x}-\boldsymbol{y})-\boldsymbol{G}_{1}
(\boldsymbol{\widetilde{x}}-\boldsymbol{y})|
\leq C\frac{\exp{(-\Im(\kappa_{1})|\boldsymbol{y}|)}}{|\boldsymbol{y}|}
|\boldsymbol{x }-\boldsymbol{\widetilde{x}}|.
\end{eqnarray*}
Therefore
\begin{align}\label{EL01-7}
|I_{5}|
&\leq\int_{n}^{+\infty}\int_{-\infty}^{+\infty}
|\boldsymbol{\varphi}(\boldsymbol{x},y_{1},y_{2})-\boldsymbol{\varphi}
(\boldsymbol{\widetilde{x}},y_{1},y_{2})|
{\rm d}y_{1}{\rm d}y_{2}\nonumber\\
&\leq C\int_{n}^{+\infty}\int_{-\infty}^{+\infty}
[|\boldsymbol{G}_{1}(\boldsymbol{x}-\boldsymbol{y})-\boldsymbol{G}_{1}
(\boldsymbol{\widetilde{x}}-\boldsymbol{y})|
\cdot|\boldsymbol{\Psi}(\boldsymbol{y})|_{y_{3}=f(y_{1}, y_{2})}]
{\rm d}y_{1}{\rm d}y_{2}\nonumber\\
&\leq
C(|\boldsymbol{x}-\boldsymbol{\widetilde{x}}|)\sup\limits_{\boldsymbol{y}\in
S}|\boldsymbol{\Psi}(\boldsymbol{y})|\int_{n}^{+\infty}\int_{-\infty}^{+\infty}
\bigg(\frac{\exp{(-\Im(\kappa_{1})|\boldsymbol{y}|)}}{|\boldsymbol{y}|}\bigg|_{
y_{3}=f(y_{1}, y_{2})}\bigg)
{\rm d}y_{1}{\rm d}y_{2}\nonumber\\
&\leq
C(|\boldsymbol{x}-\boldsymbol{\widetilde{x}}|)\|\boldsymbol{\Psi}\|_{C^{0,\alpha
}(S)}
\left(\frac{2}{\Im{(\kappa_{1})}}\right)^{2}
\left(\frac{1}{n}\exp{\left(-\frac{n}{2}\Im{(\kappa_{1})}\right)}\right).
\end{align}
Similarly, for $j=6,7,8$, we also have
\begin{eqnarray}\label{EL01-8}
|I_{j}|
\leq
C(|\boldsymbol{x}-\boldsymbol{\widetilde{x}}|)\|\boldsymbol{\Psi}\|_{C^{0,\alpha
}(S)}
\left(\frac{1}{n}\exp{\big(-\frac{n}{2}\Im{(\kappa_{1})}\big)}\right).
\end{eqnarray}
Combining \eqref{EL01-6}--\eqref{EL01-8} and noting
$0<\alpha\leq 1$, we obtain
\begin{align}\label{EL01-9}
&\frac{|((T-T_{n})\boldsymbol{\Psi})(\boldsymbol{x})-((T-T_{n})\boldsymbol{\Psi}
)(\boldsymbol{\widetilde{x}})|}
{|\boldsymbol{x}-\boldsymbol{\widetilde{x}}|^{\alpha}}\nonumber\\
&\leq \sum_{j=5}^{8}|I_{j}|
\leq
C(|\boldsymbol{x}-\boldsymbol{\widetilde{x}}|^{1-\alpha})\|\boldsymbol{\Psi}
\|_{C^{0,\alpha}(S)}
\left(\frac{1}{n}\exp{\big(-\frac{n}{2}\Im{(\kappa_{1})}\big)}
\right)\nonumber\\
&\leq C(n^{1-\alpha})\|\boldsymbol{\Psi}\|_{C^{0,\alpha}(S)}
\left(\frac{1}{n}\exp{\big(-\frac{n}{2}\Im{(\kappa_{1})}\big)}
\right)\nonumber\\
&= C\|\boldsymbol{\Psi}\|_{C^{0,\alpha}(S)}
\left(n^{-\alpha}\exp{\big(-\frac{n}{2}\Im{(\kappa_{1})}\big)}\right)
\rightarrow 0
\quad\text{as}\ n\rightarrow+\infty.
\end{align}

For $0<\alpha\leq 1$, it can be deduced from \eqref{EL01-5} and
\eqref{EL01-9} that
\begin{align*}
&\|T-T_{n}\|_{C^{0,\alpha}(S)}\\
&=\sup_{\|\boldsymbol{\Psi}\|_{C^{0,\alpha}(S)}\neq0}\frac{\|(T-T_{n}
)\boldsymbol{\Psi}\|_{C^{0,\alpha}(S)}}{\|\boldsymbol{\Psi}\|_{C^{0,\alpha}(S)}}
\\
&=\sup_{\|\boldsymbol{\Psi}\|_{C^{0,\alpha}(S)}\neq0}\frac{1}{\|\boldsymbol{\Psi
}\|_{C^{0,\alpha}(S)}}
\bigg[\|(T-T_{n})\boldsymbol{\Psi}\|_{\infty}\\
&\hspace{3cm}+\sup\limits_{\substack{
\boldsymbol
{x},\boldsymbol{\widetilde{x}}\in S\\
\boldsymbol{x}\neq\boldsymbol{\widetilde{x}}}}\frac{|((T-T_{n})\boldsymbol{\Psi}
)(\boldsymbol{x} )-((T-T_{n})\boldsymbol{\Psi})(\boldsymbol{\widetilde{x}})|}
{|\boldsymbol{x}-\boldsymbol{\widetilde{x}}|^{\alpha}}\bigg]\\
&\leq C \left(n^{-\alpha}\exp{\big(-\frac{n}{2}\Im{(\kappa_{1})}\big)}
\right)\rightarrow 0
\quad\text{as}\ n\rightarrow+\infty,
\end{align*}
which shows that the operator $T$ is compact on $\mathfrak{T}^{0,\alpha}(S)$.
Similarly, it can be shown from \eqref{CP2} and \eqref{CPn2} that operator $K$
is also compact on $\mathfrak{T}^{0,\alpha}(S)$.
\end{proof}

\begin{theorem}\label{ET2}
Let $\boldsymbol{E}^{s}\in \mathcal{T}_{1}, \boldsymbol{E}_{2}\in
\mathcal{T}_{2}$ have the integral representations \eqref{E1}--\eqref{E2}
and satisfy the boundary integral equations \eqref{E1S}--\eqref{E2S} with the
continuity conditions \eqref{E1bcs}. Then $(\boldsymbol{E}_{1},
\boldsymbol{E}_{2})$ are the solutions of
Problem \ref{SP}.
\end{theorem}

\begin{proof}
We only show the proof for the field $\boldsymbol{E}_{1}$. If the field
$\boldsymbol{E}^{s}\in \mathcal{T}_{1}$ has the integral representation
\eqref{E1}, then we
have
\begin{align}\label{ET2-1}
\boldsymbol{E}^{s}(\boldsymbol{x})
&=\int_{S}
\big\{[\mathrm{i}\omega\mu_{1}\boldsymbol{G}_{1}(\boldsymbol{x}-\boldsymbol{y})]
\cdot[\boldsymbol{\nu}_{S}(\boldsymbol{y})\times\boldsymbol{H}_{1}(\boldsymbol{y
})]\notag\\
&\qquad+[\nabla_{\boldsymbol{x}}\times\boldsymbol{G}_{1}(\boldsymbol{x}
-\boldsymbol { y } )
]\cdot[\boldsymbol{\nu}_{S}(\boldsymbol{y})\times\boldsymbol{E}_{1}(\boldsymbol{
y})]\big\}{\rm d}s_{\boldsymbol{y}}\nonumber\\
& +\int_{\Gamma}
\big\{[\mathrm{i}\omega\mu_{1}\boldsymbol{G}_{1}(\boldsymbol{x}-\boldsymbol{y})]
\cdot[\boldsymbol{\nu}_{\Gamma}(\boldsymbol{y})\times\boldsymbol{H}_{1}
(\boldsymbol{y})]\big\}{\rm d}s_{\boldsymbol{y}},
\quad\boldsymbol{x}\in \Omega_{1}.
\end{align}
It is easy to verify
that $\boldsymbol{\nu}_{\Gamma}(\boldsymbol{y})\times\boldsymbol{E}_{1}
(\boldsymbol{y })\big|_{\Gamma}=0$, i.e., $\boldsymbol E_1=\boldsymbol
E^s+\boldsymbol E^i$ satisfies the boundary condition (ii) of Problem \ref{SP}.

Noting that for any $\boldsymbol{x}\in \Omega_{1}$ and
$\boldsymbol{y}\in S\cup \Gamma$, we have
$\boldsymbol{x}\neq\boldsymbol{y}$. Taking double curl of
\eqref{ET2-1}, multiplying \eqref{ET2-1} by
$-\kappa_{1}^{2}=-\omega^{2}\mu_{1}(\varepsilon_{1}+\mathrm{i}\frac{\sigma_{1}}{
\omega})$, and adding the resulting two equations with the aid of \eqref{G1eq},
we obtain
\begin{align}\label{ET2-2}
&\nabla\times(\nabla\times\boldsymbol{E}^{s}(\boldsymbol{x}))
-\kappa_{1}^{2}\boldsymbol{E}^{s}(\boldsymbol{x})\nonumber\\
=&\int_{S}
\big\{\mathrm{i}\omega\mu_{1}[\nabla_{\boldsymbol{x}}\times\nabla_{\boldsymbol{x
} } \times\boldsymbol{G}_{1}(\boldsymbol{x}-\boldsymbol{y})
-\kappa_{1}^{2}\boldsymbol{G}_{1}(\boldsymbol{x}-\boldsymbol{y})]
\cdot[\boldsymbol{\nu}_{S}(\boldsymbol{y})\times\boldsymbol{H}_{1}(\boldsymbol{y
})]\nonumber\\
&+[\nabla_{\boldsymbol{x}}\times(\nabla_{\boldsymbol{x}}\times\nabla_{
\boldsymbol{x}}\times\boldsymbol{G}_{1}(\boldsymbol{x}-\boldsymbol{y})
-\kappa_{1}^{2}\boldsymbol{G}_{1}(\boldsymbol{x}-\boldsymbol{y}))]
\cdot[\boldsymbol{\nu}_{S}(\boldsymbol{y})\times\boldsymbol{E}_{1}(\boldsymbol{y
})]\big\}{\rm d}s_{\boldsymbol{y}}\nonumber\\
& +\int_{\Gamma}
\big\{\mathrm{i}\omega\mu_{1}[\nabla_{\boldsymbol{x}}\times\nabla_{\boldsymbol{x
} } \times\boldsymbol{G}_{1}(\boldsymbol{x}-\boldsymbol{y})
-\kappa_{1}^{2}\boldsymbol{G}_{1}(\boldsymbol{x}-\boldsymbol{y})]
\cdot[\boldsymbol{\nu}_{\Gamma}(\boldsymbol{y})\times\boldsymbol{H}_{1}
(\boldsymbol{y})]\big\}{\rm d}s_{\boldsymbol{y}}\nonumber\\
=&0,\quad\boldsymbol{x}\in \Omega_{1}.
\end{align}
It follows from \eqref{Eieq} and \eqref{ET2-2} that
\begin{align*}
&\nabla\times(\nabla\times\boldsymbol{E}_{1}(\boldsymbol{x}))
-\kappa_{1}^{2}\boldsymbol{E}_{1}(\boldsymbol{x})\\
=&[\nabla\times\nabla\times\boldsymbol{E}^{s}(\boldsymbol{x})
-\kappa_{1}^{2}\boldsymbol{E}^{s}(\boldsymbol{x})]
+[\nabla\times\nabla\times\boldsymbol{E}^{i}(\boldsymbol{x})
-\kappa_{1}^{2}\boldsymbol{E}^{i}(\boldsymbol{x})]\\
=&\mathrm{i}
\omega\mu_{1}\boldsymbol{J}_{cs}(\boldsymbol{x}), \quad\boldsymbol{x}\in\
\Omega_{1}.
\end{align*}
Furthermore, with the help of Lemma \ref{EL01} and \eqref{ET2-1},
we deduce that
\begin{align}\label{ET2-4}
|\boldsymbol{E}^{s}(\boldsymbol{x})|
\leq & C\bigg[\int_{S}
|\boldsymbol{G}_{1}(\boldsymbol{x}-\boldsymbol{y})|
\cdot|\boldsymbol{\nu}_{S}(\boldsymbol{y})\times\boldsymbol{H}_{1}(\boldsymbol{y
})|{\rm d}s_{\boldsymbol{y}}\notag\\
& \qquad
+\int_{S}|\nabla_{\boldsymbol{x}}\times\boldsymbol{G}_{1}(\boldsymbol{x}
-\boldsymbol{y})|
\cdot|\boldsymbol{\nu}_{S}(\boldsymbol{y})\times\boldsymbol{E}_{1}(\boldsymbol{y
})|{\rm d}s_{\boldsymbol{y}}\nonumber\\
&\qquad +\int_{\Gamma}
|\boldsymbol{G}_{1}(\boldsymbol{x}-\boldsymbol{y})|
\cdot|\boldsymbol{\nu}_{\Gamma}(\boldsymbol{y})\times\boldsymbol{H}_{1}
(\boldsymbol{y})|{\rm d}s_{\boldsymbol{y}}\bigg]\nonumber\\
\leq & C\bigg[
\|\boldsymbol{\nu}_{S}\times\boldsymbol{H}_{1}\|_{C^{0,\alpha}(S)}\int_{S}
|\boldsymbol{G}_{1}(\boldsymbol{x}-\boldsymbol{y})|{\rm
d}s_{\boldsymbol{y}}\notag\\
&\qquad
+\|\boldsymbol{\nu}_{S}\times\boldsymbol{E}_{1}\|_{C^{0,\alpha}(S)}\int_{ S }
|\nabla_{\boldsymbol{x}}\times\boldsymbol{G}_{1}(\boldsymbol{x}-\boldsymbol{y})|
{\rm d}s_{\boldsymbol{y}}\nonumber\\
&\quad \quad
+\|\boldsymbol{\nu}_{\Gamma}\times\boldsymbol{H}_{1}\|_{C^{0,\alpha}(\Gamma)}
\int_{\Gamma}
|\boldsymbol{G}_{1}(\boldsymbol{x}-\boldsymbol{y})|
{\rm d}s_{\boldsymbol{y}}\bigg]\nonumber\\
\leq & C\bigg[\lim_{n\rightarrow+\infty}\int_{S_{n}}
|\boldsymbol{G}_{1}(\boldsymbol{x}-\boldsymbol{y})|{\rm d}s_{\boldsymbol{y}}
+\lim_{n\rightarrow+\infty}\int_{S_{n}}|\nabla_{\boldsymbol{x}}\times\boldsymbol
{G}_{1}(\boldsymbol{x}-\boldsymbol{y})|
{\rm d}s_{\boldsymbol{y}}\notag\\
&\qquad +\int_{\Gamma}
|\boldsymbol{G}_{1}(\boldsymbol{x}-\boldsymbol{y})|
{\rm d}s_{\boldsymbol{y}}\bigg].
\end{align}
For each fixed $n\geq 1$, as $|\boldsymbol{x}|\to +\infty$, by Lemma \ref{EL02},
we have
\begin{align}\label{ET2-5}
\int_{S_{n}}|\boldsymbol{G}_{1}(\boldsymbol{x}-\boldsymbol{y})|{\rm
d}s_{\boldsymbol{y}}&\leq C\int_{S_{n}}\bigg|\frac{\exp{(\frac{1}{2}\mathrm{i}
\kappa_{1}|\boldsymbol{x} -\boldsymbol{y}|)}}{|\boldsymbol{
x}-\boldsymbol{y}|}\exp{\big(\frac{1}{2}\mathrm{i}\kappa_{1}|\boldsymbol{x}
-\boldsymbol{y}|\big)}\bigg|{\rm d}s_{\boldsymbol{y}}
\nonumber\\
&\leq
C\frac{\exp{(-\frac{1}{2}\Im(\kappa_{1})|\boldsymbol{x}|)}}{|\boldsymbol{x}|}
\int_{S_{n}} \exp{\big(-\frac{1}{2}\Im(\kappa_{1})|\boldsymbol{y}|\big)}{\rm
d}s_{\boldsymbol{y}}\nonumber\\
&\leq
C\frac{\exp{(-\frac{1}{2}\Im(\kappa_{1})|\boldsymbol{x}|)}}{|\boldsymbol{x}|}
\left(\int_{0}^{n} \exp{\big(-\frac{1}{4}\Im(\kappa_{1})y_{1}\big)}{\rm d}
y_{1}\right)^{2}\nonumber\\
&\leq
C \frac{\exp{(-\frac{1}{2}\Im(\kappa_{1})|\boldsymbol{x}|)}}{|\boldsymbol{x}|}
\left(1-\exp{\big(-\frac{n}{4}\Im{(\kappa_{1})}\big)}\right)^{2}
\end{align}
and
\begin{align}\label{ET2-6}
\int_{S_{n}}
|\nabla_{\boldsymbol{x}}\times\boldsymbol{G}_{1}(\boldsymbol{x}-\boldsymbol{y})|
{\rm d}s_{\boldsymbol{y}}
\leq C
\frac{\exp{(-\frac{1}{2}\Im(\kappa_{1})|\boldsymbol{x}|)}}{|\boldsymbol{x}|}
\left(1-
\exp{\big(-\frac{n}{4}\Im(\kappa_{1})\big)}\right)^{2}.
\end{align}
Similarly, we can obtain
\begin{align}\label{ET2-7}
\int_{\Gamma}
|\boldsymbol{G}_{1}(\boldsymbol{x}-\boldsymbol{y})|{\rm d}s_{\boldsymbol{y}}
&\leq C \frac{\exp{(-\Im(\kappa_{1})|\boldsymbol{x}|)}}{|\boldsymbol{x}|}
\int_{\Gamma}
\bigg|\exp{(-\mathrm{i}\kappa_{j}\boldsymbol{\hat{x}}\cdot\boldsymbol{y})}[
\boldsymbol{I}-\boldsymbol{\hat{x}}\boldsymbol{\hat{x}}]
+\mathcal O\left(\frac{1}{|\boldsymbol{x}|}\right)\boldsymbol{\hat{I}}\bigg|{\rm
d}s_{\boldsymbol{y}}\nonumber\\
&\leq C \frac{\exp{(-\Im(\kappa_{1})|\boldsymbol{x}|)}}{|\boldsymbol{x}|}.
\end{align}
Combining \eqref{ET2-4}--\eqref{ET2-7}, we have for $\Im{(\kappa_{1})}>0$ that
\[
|\boldsymbol{E}^{s}(\boldsymbol{x})|=\mathcal
O\left(\frac{\exp{(-\frac{1}{2}\Im{(\kappa_{1})}
|\boldsymbol{x}|)}}{|\boldsymbol{x}|}\right)\quad \text{as}\
|\boldsymbol{x}|\to +\infty
\]
and
\begin{align*}
\int_{\partial
B_{r}^{+}}|\boldsymbol{E}^{s}|^{2}{\rm d}{s_{\boldsymbol{x}}}
&\leq C \int_{\partial
B_{r}^{+}}\frac{\exp{(-\Im(\kappa_{1})|\boldsymbol{x}|)}}{|\boldsymbol{x}|^{2}}
{\rm d}{s_{\boldsymbol{x}}}\\
&\leq
C\left(\frac{\exp{(-\Im{(\kappa_{1})}r)}}{r^{2}}4\pi{r^{2}}\right)=
C\exp{(-\Im{(\kappa_{1})}r)}\rightarrow 0\quad \text{as}\ r\to +\infty,
\end{align*}
where $C$ is a positive constant independent of $r$.

Similarly, we can also show that
\[
\lim_{r\to +\infty}\int_{\partial
B_{r}^{+}}|\boldsymbol{H}^{s}|^{2}{\rm d}{s_{\boldsymbol{x}}}
=\lim_{r\to +\infty}\int_{\partial
B_{r}^{-}}|\boldsymbol{E}_{2}|^{2}{\rm d}{s_{\boldsymbol{x}}}
=\lim_{r\to +\infty}\int_{\partial
B_{r}^{-}}|\boldsymbol{H}_{2}|^{2}{\rm d}{s_{\boldsymbol{x}}}=0,
\]
which complete the proof.
\end{proof}

It can be seen from Theorems \ref{ET1} and \ref{ET2} that there exits a
solution of Problem \ref{SP} by using the boundary integral
equation method. To prove the uniqueness, it suffices to show that
$\boldsymbol{E}_{1}=\boldsymbol{E}^{s}$ and
$\boldsymbol{E}_{2}$ vanish identically in $\overline{\Omega}_{1}$ and
$\overline{\Omega}_{2}$ if $\boldsymbol{E}^{i}=0$.
For the sake of brevity  for the proof, we consider the homogeneous Maxwell's
equations
\begin{equation}\label{Eeq01}
\nabla\times(\nabla\times\boldsymbol{E}_j)
-\kappa_j^{2}\boldsymbol{E}_j=0 \quad \text{in}\ \Omega_j,
\end{equation}
along with the boundary condition
\begin{equation}\label{EBCS01}
\boldsymbol{\nu}_{\Gamma}\times{\boldsymbol{E}_{1}}=0\quad \text{on}\ \Gamma,
\end{equation}
and the continuity conditions
\begin{equation}\label{EBCS0}
\boldsymbol{\nu}_{S}\times{\boldsymbol{E}_{1}}=\boldsymbol{\nu}_{S}\times{
\boldsymbol{E}_{2}}, \quad
\boldsymbol{\nu}_{S}\times\boldsymbol{H}_{1}=\boldsymbol{\nu}_{S}
\times\boldsymbol{H}_{2}\quad \text{on}\ S,
\end{equation}
and the radiation conditions
\begin{align}\label{EsIRC0}
&\lim_{r\to +\infty}\int_{\partial B_{r}^{+}}|\boldsymbol{E}_{1}|^{2}{\rm
d}s=\lim_{r\to +\infty}\int_{\partial
B_{r}^{+}}|\boldsymbol{H}_{1}|^{2}{\rm d}s\notag\\
= &\lim_{r\rightarrow+\infty}\int_{\partial
B_{r}^{-}}|\boldsymbol{E}_{2}|^{2}{\rm
d}s=\lim_{r\rightarrow+\infty}\int_{\partial
B_{r}^{-}}|\boldsymbol{H}_{2}|^{2}{\rm d}s=0.
\end{align}

\begin{theorem}\label{ET3}
Let $(\boldsymbol{E}_{1}, \boldsymbol{E}_{2})$ be the solutions of
the problem \eqref{Eeq01}--\eqref{EsIRC0}. Then $(\boldsymbol{E}_{1},
\boldsymbol{E}_{2})$ vanish identically.
\end{theorem}

\begin{proof}
Denote $\Omega_{r}=(B_{r}\cap \Omega_{1})$ with boundary $\partial
\Omega_{r}=\partial B_{r}^{+}\cup \Gamma\cup S_{r}$,
where $\partial B_{r}^{+}=\partial B_{r}\cap \Omega_{1}$ and
$S_{r}=S\cap B_{r}$.
For each fixed $\boldsymbol{x}\in \Omega_{r}$,
applying the vector Green first theorem to $\boldsymbol{E}_{1}$ in
$\Omega_{r}$, we have
\begin{align}\label{ET3-1}
&\int_{\Omega_{r}}\big[|\nabla\times\boldsymbol{E}_{1}|^{2}
-\boldsymbol{E}_{1}\cdot(\nabla\times\nabla\times\overline{\boldsymbol{E}}_{1}
)\big]
{\rm d}{\boldsymbol{x}}\nonumber\\
= & \int_{\partial
\Omega_{r}}\boldsymbol{\nu}\cdot[\boldsymbol{E}_{1}
\times(\nabla\times\overline{\boldsymbol{E}}_{1})] {\rm
d}{s_{\boldsymbol{x}}}
= \int_{\partial
\Omega_{r}}(\nabla\times\overline{\boldsymbol{E}}_{1})\cdot[\boldsymbol{\nu}
\times\boldsymbol{E}_{1}] {\rm d}{s_{\boldsymbol{x}}}\nonumber\\
= & \mathrm{i}\omega\mu_{1}\left(\int_{\partial
B_{r}^{+}}+\int_{\Gamma}+\int_{S_{r}}\right)
\overline{\boldsymbol{H}}_{1}\cdot(\boldsymbol{\nu}\times\boldsymbol{E}_{1})
{\rm d}{s_{\boldsymbol{x}}},
\end{align}
where $\boldsymbol{\nu}=\boldsymbol{\nu}(\boldsymbol{x})$ stands for the unit
normal vector at
$\boldsymbol{x}\in\partial \Omega_{r}$ pointing out of $\Omega_{r}$.
Letting $r\to +\infty$, we have from \eqref{EBCS01}, \eqref{EsIRC0}, and
\eqref{ET3-1} that
\begin{eqnarray}\label{ET3-2}
\int_{\Omega_{1}}\big[|\nabla\times\boldsymbol{E}_{1}|^{2}
-\boldsymbol{E}_{1}\cdot(\nabla\times\nabla\times\overline{\boldsymbol{E}}_{1}
)\big]{\rm d}{\boldsymbol{x}}
=-\mathrm{i}\omega\mu_{1}\int_{S}
\overline{\boldsymbol{H}}_{1}\cdot(\boldsymbol{\nu}_{S}\times\boldsymbol{E}_{1})
{\rm d}{s_{\boldsymbol{x}}},
\end{eqnarray}
where $\boldsymbol{\nu}_{S}=\boldsymbol{\nu}_{S}(\boldsymbol{x})$
denotes the unit normal vector at $\boldsymbol{x} \in S$ pointing from
region $\Omega_{2}$ to region $\Omega_{1}$.

Using \eqref{ET3-2} and \eqref{Eeq01} yields
\begin{align}\label{ET3-3}
&\int_{\Omega_{1}}\big[|\nabla\times\boldsymbol{E}_{1}|^{2}
-\boldsymbol{E}_{1}\cdot(\nabla\times\nabla\times\overline{\boldsymbol{E}}_{1}
)\big]{\rm d}{\boldsymbol{x}}\nonumber\\
= & \int_{\Omega_{1}}\left(|\nabla\times\boldsymbol{E}_{1}|^{2}
-\omega^{2}\mu_{1}\varepsilon_{1}|\boldsymbol{E}_{1}|^{2}
+\mathrm{i}\omega\mu_{1}\sigma_{1}|\boldsymbol{E}_{1}|^{2}\right){\rm
d}{\boldsymbol{x}}\nonumber\\
= & -\mathrm{i}\omega\mu_{1}\int_{S}
\overline{\boldsymbol{H}}_{1}\cdot(\boldsymbol{\nu}_{S}\times\boldsymbol{E}_{1})
{\rm d}{s_{\boldsymbol{x}}},
\end{align}
which gives by taking the imaginary part of \eqref{ET3-3} that
\begin{eqnarray}\label{ET3-4}
-\Re\bigg[\int_{S}
\overline{\boldsymbol{H}}_{1}\cdot(\boldsymbol{\nu}_{S}\times\boldsymbol{E}_{1})
{\rm d}{s_{\boldsymbol{x}}}\bigg]
=\sigma_{1}\int_{\Omega_{1}}|\boldsymbol{E}_{1}|^{2}{\rm
d}{\boldsymbol{x}}\geq0.
\end{eqnarray}
Similarly, we may show that
\begin{eqnarray}\label{ET3-5}
\Re\bigg[\int_{S}
\overline{\boldsymbol{H}}_{2}\cdot(\boldsymbol{\nu}_{S}\times\boldsymbol{E}_{2})
{\rm d}{s_{\boldsymbol{x}}}\bigg]
=\sigma_{2}\int_{\Omega_{2}}|\boldsymbol{E}_{2}|^{2}{\rm
d}{\boldsymbol{x}}\geq0.
\end{eqnarray}
Noting the continuity conditions \eqref{EBCS0} and
$\boldsymbol{H}_{j}=\boldsymbol{\nu}_{S}\times(\boldsymbol{H}_{j}
\times\boldsymbol{\nu}_{S})+
(\boldsymbol{\nu}_{S}\cdot\boldsymbol{H}_{j})\boldsymbol{\nu}_{S}, j=1, 2$ on
$S$, we have
$\overline{\boldsymbol{H}}_{1}\cdot(\boldsymbol{\nu}_{S}\times\boldsymbol{E}
_{1}
)=\overline{\boldsymbol{H}}_{2}\cdot(\boldsymbol{\nu}_{S}\times\boldsymbol{E}_{2
})$ on $S$, and
\begin{eqnarray}\label{ET3-6}
\int_{S}
\overline{\boldsymbol{H}}_{1}\cdot(\boldsymbol{\nu}_{S}\times\boldsymbol{E}_{1})
{\rm d}{s_{\boldsymbol{x}}}
=\int_{S}
\overline{\boldsymbol{H}}_{2}\cdot(\boldsymbol{\nu}_{S}\times\boldsymbol{E}_{2})
{\rm d}{s_{\boldsymbol{x}}}.
\end{eqnarray}
It follows immediately from combining \eqref{ET3-4}--\eqref{ET3-6} and
$\sigma_{j}>0$ that
\[
\int_{{\Omega}_{1}}|\boldsymbol{E}_{1}|^{2}{\rm
d}{\boldsymbol{x}}=\int_{{\Omega}_{2}}|\boldsymbol{E}_{2}|^{2}{\rm
d}{\boldsymbol{x}}=0,
\]
which implies that $\boldsymbol E_1=0$ in $\Omega_1$ and $\boldsymbol E_2=0$ in
$\Omega_2$.
\end{proof}

\section{Uniqueness of the inverse problem}

This section addresses the uniqueness of the inverse hybrid surface scattering
problem. For the given incident field, we show that the obstacle and the
unbounded rough surface can be uniquely determined by  the tangential trace
of the electric field
$\boldsymbol{\nu}_{\Gamma_{H}}\times\boldsymbol{E}_{1}|_{\Gamma_{H}}$,
where $\Gamma_{H}=\{\boldsymbol x\in \mathbb{R}^{3}|\ x_{3}=H\}$ is a
plane surface above the obstacle and unbounded rough surface and
$\boldsymbol{\nu}_{\Gamma_{H}}=(0,0,1)^{\top}$.

Let $\widetilde S\in C^2$ be an unbounded rough surface which divides
$\mathbb R^3$ into the upper half space $\widetilde\Omega_1^+$ and the lower
half space $\widetilde\Omega_2$. Let $\widetilde
D\subset\subset\widetilde\Omega_1^+$ be a bounded domain with the boundary
$\widetilde\Gamma\in C^2$.
Define $\widetilde\Omega_1=\widetilde\Omega_1^+\setminus\overline{\widetilde
D}$. Let $(\widetilde{\boldsymbol{E}}_{1},
\widetilde{\boldsymbol{E}}_{2})$ be the unique solutions of Problem 2.1 with
the hybrid surface
$(D, S)$ replaced by $(\widetilde{D}, \widetilde S)$
but for the same incident field $\boldsymbol E^i$ satisfying \eqref{PSEi}.
The point dipole source is assumed to be located at $\boldsymbol{x}_s\in
\Omega_{1}\cap \widetilde{\Omega}_{1}$.

\begin{theorem}\label{ET4}
Assume that
$\boldsymbol{\nu}_{\Gamma_{H}}\times\boldsymbol{E}_{1}|_{\Gamma_{H}}=
\boldsymbol{\nu}_{\Gamma_{H}}\times\boldsymbol{\widetilde{E}}_{1}|_{\Gamma_{H}}$
, then $D=\widetilde{D}, S=\widetilde{S}$.
\end{theorem}

\begin{proof}
We prove it by contradiction and assume that $D\neq \widetilde{D}$, $S\neq
\widetilde{S}$. The problem geometry is shown in  Figure \ref{pg_2}. Let
$\boldsymbol{\widehat{E}}=\boldsymbol{E}_{1}-\boldsymbol{\widetilde{E}}_{1}$,
then $\boldsymbol{\widehat{E}}$ satisfies Maxwell's equation
\[
\nabla\times\nabla\times
\boldsymbol{\widehat{E}}-\kappa_{1}^{2}\boldsymbol{\widehat{E}}=0\quad\text{in
} \ \Omega_{1}\cap \widetilde{\Omega}_{1}.
\]
By the assumption $
\boldsymbol{\nu}_{\Gamma_{H}}\times\boldsymbol{E}_{1}|_{\Gamma_{H}}=
\boldsymbol{\nu}_{\Gamma_{H}}\times\boldsymbol{\widetilde{E}}_{1}|_{\Gamma_{H}}$
and the uniqueness result for the direct scattering problem,
it follows that $\boldsymbol{E}_{1}(\boldsymbol{x})=
\boldsymbol{\widetilde{E}}_{1}(\boldsymbol{x})$ for all
$\boldsymbol{x}\in \Omega_{H}=\{\boldsymbol x\in \mathbb{R}^{3}|: x_{3}\geq
H\}$.
By the analytic continuation,
we get that $\boldsymbol{E}_{1}(\boldsymbol{x})=
\boldsymbol{\widetilde{E}}_{1}(\boldsymbol{x})$ for all $\boldsymbol{x}\in
\Omega_{1}\cap \widetilde{\Omega}_{1}$. Since
$\boldsymbol{\widehat{E}}\in C^{2}(\Omega_{1}\cap \widetilde{\Omega}_{1})\cap
C^{0,\alpha}(\overline{\Omega_{1}\cap \widetilde{\Omega}_{1}})$, we have
\begin{equation*}\label{ET4-2}
\boldsymbol{\widehat{E}}(\boldsymbol{x})
=\boldsymbol{E}_{1}(\boldsymbol{x})-\boldsymbol{\widetilde{E}}_{1}(\boldsymbol{x
})
= 0,\quad \boldsymbol{x}\in\ \overline{\Omega_{1}\cap
\widetilde{\Omega}_{1}}.
\end{equation*}
In particular, we have
\begin{equation}\label{ET4-3}
\boldsymbol{\widehat{E}}\big|_{\partial (\Omega_{1}\cap \widetilde{\Omega}_{1})}
=\boldsymbol{E}_{1}\big|_{\partial (\Omega_{1}\cap \widetilde{\Omega}_{1})}
-\boldsymbol{\widetilde{E}}_{1}\big|_{\partial (\Omega_{1}\cap
\widetilde{\Omega}_{1})}=0.
\end{equation}

\begin{figure}
\centering
\includegraphics[width=0.5\textwidth]{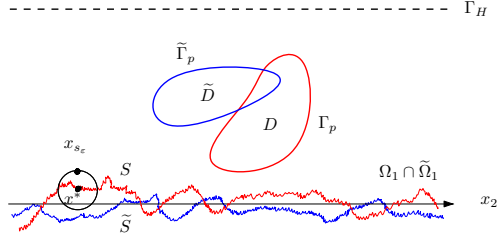}
\caption{Problem geometry of the scattering problem.}
\label{pg_2}
\end{figure}

First, we prove that the obstacle can be uniquely determined. In the case when
$D\neq \widetilde{D}$ which include $D\cap \widetilde{D}\neq
\emptyset$ and $D\cap \widetilde{D}=\emptyset$, without loss of generality,
let us denote the region between  $D$  and $D\cap \widetilde{D}$  by
$\widetilde{Q}=D\setminus\overline{(D\cap \widetilde{D})}$,
then we have $\widetilde{Q}\subset D$ and $\widetilde{Q}\nsubseteq
\widetilde{D}$
with the boundary $\partial\widetilde{Q}=\Gamma_{p}\cup \widetilde{\Gamma}_{
p}$,
where $\Gamma_{p}$ and $\widetilde{\Gamma}_{p}$ denote the part of the boundary
$\Gamma$ and $\widetilde{\Gamma}$, respectively.
Thus, from \eqref{ET4-3} and \eqref{EBCS1}, we obtain
\begin{eqnarray}\label{ET4-4}
\boldsymbol{\nu}_{\Gamma_{p}}\times\boldsymbol{\widetilde{E}}_{1}\big|_{\Gamma_{
p}}
=\boldsymbol{\nu}_{\widetilde{\Gamma}_{p}}\times\boldsymbol{\widetilde{E}}_{1}
\big|_{\widetilde{\Gamma}_{p}}=0.
\end{eqnarray}
Applying vector Green's first theorem to $\boldsymbol{\widetilde{E}}_{1}$  in
$\widetilde{Q}$, we have from \eqref{ET4-4} that
\begin{align*}
\int_{\widetilde{Q}}\left(|\nabla\times\boldsymbol{\widetilde{E}}_{1}|^{2}
-\boldsymbol{\widetilde{E}}_{1}\cdot(\nabla\times\nabla\times\overline{
\boldsymbol{\widetilde{E}}_{1}})\right){\rm d}{\boldsymbol{x}}
=\int_{\partial \widetilde{Q}}\boldsymbol{\nu}
\cdot[\boldsymbol{\widetilde{E}}_{1}\times(\nabla\times\overline{\boldsymbol{
\widetilde{E}}_{1}})] {\rm d}{s_{\boldsymbol{x}}}\nonumber\\
=\int_{\partial
\widetilde{Q}}(\nabla\times\overline{\boldsymbol{\widetilde{E}}_{1}})
\cdot[\boldsymbol{\nu}\times\boldsymbol{\widetilde{E}}_{1}] {\rm
d}{s_{\boldsymbol{x}}}=0.
\end{align*}
On the other hand, note that the incident field is a point
dipole source located at $\boldsymbol{x}_s \in \Omega_{1}\cap
\widetilde{\Omega}_{1}$, then we have
$\mathrm{i}\omega\mu_{1}\boldsymbol{\tilde{J}}_{cs}(\boldsymbol{x})=\boldsymbol{
I}\delta(\boldsymbol{x}-\boldsymbol{x}_s)=0$ in $\widetilde{Q}$.
By \eqref{Eeq1}, we have
\begin{align}\label{ET4-6}
&\int_{\widetilde{Q}}\left(|\nabla\times\boldsymbol{\widetilde{E}}_{1}|^{2}
-\boldsymbol{\widetilde{E}}_{1}\cdot(\nabla\times\nabla\times\overline{
\boldsymbol{\widetilde{E}}_{1}})\right){\rm d}{\boldsymbol{x}}\nonumber\\
=&\int_{\widetilde{Q}}\left(|\nabla\times\boldsymbol{\widetilde{E}}_{1}|^{2}
-\overline{\kappa_{1}^{2}}|\boldsymbol{\widetilde{E}}_{1}|^{2}+\mathrm{i}
\omega\mu_{1}\boldsymbol{\widetilde{E}}_{1}
\cdot\overline{\boldsymbol{\tilde{J}}}_{cs}\right){\rm
d}{\boldsymbol{x}}\nonumber\\
=&\int_{\widetilde{Q}}\left(|\nabla\times\boldsymbol{\widetilde{E}}_{1}|^{2}
-\omega^{2}\mu_{1}\varepsilon_{1}|\boldsymbol{\widetilde{E}}_{1}|^{2}
+\mathrm{i}\omega\mu_{1}\sigma_{1}|\boldsymbol{\widetilde{E}}_{1}|^{2}\right){
\rm d}{\boldsymbol{x}}.
\end{align}
For $\omega\mu_{1}\sigma_{1}>0$, taking the imaginary part of \eqref{ET4-6}, we
obtain $
\int_{\widetilde{Q}}|\boldsymbol{\widetilde{E}}_{1}|^{2}{\rm
d}{\boldsymbol{x}}=0,$
which implies that $\boldsymbol{\widetilde{E}}_{1}=0$ in $\widetilde{Q}$.
It follows from Theorem \ref{ET3} and $\boldsymbol{\widetilde{E}}_{1}\in
\mathcal{T}_{1}(\widetilde{\Omega}_{1})$ that
we have $\boldsymbol{\widetilde{E}}_{1}=0$ in
$\overline{\widetilde{\Omega}}_{1}$.
This is a contradiction because the total field $\boldsymbol{\widetilde{E}}_{1}$
is a nontrivial solution of the inhomogeneous equation \eqref{Eeq1} in
$\widetilde{\Omega}_{1}$.
Hence, $D=\widetilde{D}$.

Next we show that the unbounded rough surface can also be uniquely determined.
In the case when $S\neq \widetilde{S}$ which includes $S\cap \widetilde{S}\neq
\emptyset$ and $S\cap \widetilde{S}=\emptyset$.
If $S_{1}$ is a segment of $S$,
we may assume without loss of generality that $S_{1}$ is located above
$\widetilde{S}$. Let $\boldsymbol{x}^{*} \in S_{1}$,
choose $\varepsilon>0$ such that
$\boldsymbol{x}_{\varepsilon}:=\boldsymbol{x}^{*}+\varepsilon
\boldsymbol{e}_{3}\in \Omega_{1}\cap \widetilde{\Omega}_{1}$, where
$\boldsymbol e_3=(0, 0, 1)^\top$. Assuming that the incident field is given by a
point dipole source located at
$\boldsymbol{x}_{\varepsilon}$ with the unit polarization vector
$\boldsymbol{q}_{\varepsilon}$, we take
\begin{equation}\label{ET4-12}
\boldsymbol{E}^{i}(\boldsymbol{x})=\boldsymbol{G}_{1}(\boldsymbol{x}-\boldsymbol
{x}_{\varepsilon})\boldsymbol{q}_{\varepsilon}\quad\text{in}\
\overline{\Omega_{1}^{+}\cap\widetilde{ \Omega}_{1}^{+}}.
\end{equation}
Let $S_{\varepsilon}=S_{1}\cap B_{\varepsilon}(\boldsymbol{x}^{*})\subset S$,
where $B_{\varepsilon}(\boldsymbol{x}^{*})$ denotes the sphere centered at the
origin $\boldsymbol{x}^{*}$ with radius $\varepsilon$
and $\widetilde{S}\cap B_{\varepsilon}(\boldsymbol{x}^{*})=\emptyset$.
Then, from $\boldsymbol{E}^{i}+\boldsymbol{E}^{s}=\boldsymbol{E}_{1}$ and
\eqref{ET4-12}, we have
\begin{equation}\label{ET4-13}
\|\boldsymbol{G}_{1}\boldsymbol{q}_{\varepsilon}\|_{L^{\infty}(S_{\varepsilon})}
=\|\boldsymbol{E}_{1}-\boldsymbol{E}^{s}\|_{L^{\infty}(S_{\varepsilon})}
\leq\|\boldsymbol{E}_{1}\|_{L^{\infty}(S_{\varepsilon})}
+\|\boldsymbol{E}^{s}\|_{L^{\infty}(S_{\varepsilon})}.
\end{equation}
Because $\boldsymbol{x}^{*}$ has a positive distance from $\widetilde{S}$,
the well-posedness of the direct problem implies that there exists $C_1 > 0$
(independent of $\varepsilon$) such that
$
\boldsymbol{E}_{1}\big|_{S_{\varepsilon}}
=\boldsymbol{\widetilde{E}}_{1}\big|_{S_{\varepsilon}}
$
and
$
\boldsymbol{E}^{s}\big|_{S_{\varepsilon}}
=\boldsymbol{\widetilde{E}}^{s}\big|_{S_{\varepsilon}}
$
satisfy the estimate
\begin{equation}\label{ET4-14}
\|\boldsymbol{E}_{1}\|_{L^{\infty}(S_{\varepsilon})}
+\|\boldsymbol{E}^{s}\|_{L^{\infty}(S_{\varepsilon})}
=\|\boldsymbol{\widetilde{E}}_{1}\|_{L^{\infty}(S_{\varepsilon})}
+\|\boldsymbol{\widetilde{E}}^{s}\|_{L^{\infty}(S_{\varepsilon})}
\leq C_1<+\infty.
\end{equation}
It follows from \eqref{ET4-13} and \eqref{ET4-14} that
\begin{equation}\label{ET4-15}
\|\boldsymbol{G}_{1}\boldsymbol{q}_{\varepsilon}\|_{L^{\infty}(S_{\varepsilon})}
\leq C_1.
\end{equation}
This is a contradiction because the left-hand side of the above inequality
\eqref{ET4-15} goes to infinity as $\varepsilon\to 0$. Hence, $S=\widetilde{S}$.
\end{proof}

\section{Local stability}

In this section, we present a local stability result. Let us begin with the
calculation of domain derivative which plays an important role in the stability
analysis.

Let $\mathcal{I}: \mathbb{R}^{3}\rightarrow\mathbb{R}^{3}$  be the identity
mapping and let $\theta: \Gamma\cup S\rightarrow\mathbb{R}^{3}$ be an
admissible perturbation, where $\theta$ is assumed to be an admissible
perturbation in $C^{2}(\Gamma\cup S, \mathbb{R}^{3})$ and has a compact support.
For $\theta\in C^{2}(\Gamma\cup S, \mathbb{R}^{3})$,
we can extend the definition of function $\theta(\boldsymbol{x})$ to
$\overline{\Omega}_{j}$ by satisfying:
$\theta(\boldsymbol{x})\in C^{2}(\Omega_{j}, \mathbb{R}^{3}) \cap C
(\overline{\Omega}_{j})$;
$\mathcal{I}+{\theta}: \Omega_{j}\rightarrow \Omega_{j,\theta}, j=1, 2$.
Here the region $\Omega_{j,\theta}$ bounded by $\Gamma_{\theta}$ and
$S_{\theta}$, where
\begin{eqnarray*}\label{GaS}
\Gamma_{\theta}=\{\boldsymbol{x}+\theta(\boldsymbol{x}):\boldsymbol{x}\in
\Gamma\},\quad S_{\theta}=\{\boldsymbol{x}+\theta(\boldsymbol{x}):\boldsymbol{x}
\in S\}.
\end{eqnarray*}
Let $\theta(\boldsymbol{x})=(\theta_{1}(\boldsymbol{x}),
\theta_{2}(\boldsymbol{x}), \theta_{3}(\boldsymbol{x}))^{\top}$.
Clearly, $\Omega_{j,\theta}$ is an admissible perturbed configuration of the
reference region $\Omega_{j}$. Note that
$\Omega_{j,\boldsymbol{0}}=\Omega_{j}$, $\Gamma_{\boldsymbol{0}}=\Gamma$,
and $S_{\boldsymbol{0}}=S$.
According to Theorem \ref{ET3}, there exist the unique solutions
$(\boldsymbol{E}_{1,\theta}, \boldsymbol{E}_{2,\theta})$
to Problem \ref{SP} corresponding to the region $\Omega_{j,\theta}$
for any small enough $\theta$. Note that this function
$\boldsymbol{E}_{j,\theta}=\boldsymbol{E}_{j}(\theta, \boldsymbol{x})$
cannot be differentiated with respect to $\theta$ in the classical sense.
For this reason, we adopt the following concept of a domain derivative.

Denote by
\[
 \boldsymbol{E}'_{j}=\frac{\partial\boldsymbol{E}_{j, \theta}}{\partial
\theta}(0)\boldsymbol{p}
\]
the domain derivative of $\boldsymbol{E}_{j,\theta}$ at $\theta=0$ in the
direction $\boldsymbol{p}(\boldsymbol{x})
=(p_{1}(\boldsymbol{x}), p_{2}(\boldsymbol{x}), p_{3}(\boldsymbol{x}))^{\top}\in
C^{2}(\Gamma\cup S, \mathbb{R}^{3})$.
Define a nonlinear map
\begin{eqnarray*}\label{NoNM}
Y : \Gamma_{\theta}\cup S_{\theta}\rightarrow
\boldsymbol{\nu}_{\Gamma_{H}}\times \boldsymbol{E}_{1,\theta}|_{\Gamma_{H}}.
\end{eqnarray*}
The domain derivative of the operator $Y$ on the boundary $\Gamma\cup S$ along
the direction $\boldsymbol{p}$ is defined by
\begin{eqnarray*}\label{NoNMF}
Y'(\Gamma\cup S, \boldsymbol{p}):=\boldsymbol{\nu}_{\Gamma_{H}}\times
\boldsymbol{E}'_{1}|_{\Gamma_{H}}.
\end{eqnarray*}

We introduce the notations
\begin{eqnarray*}\label{nt}
\boldsymbol{V}_{\Gamma_{\tau}}=\boldsymbol{\nu}_{\Gamma}\times(\boldsymbol{V}
\times\boldsymbol{\nu}_{\Gamma}),\quad
\boldsymbol{V}_{\Gamma_{\nu}}=\boldsymbol{\nu}_{\Gamma}\cdot\boldsymbol{V},\quad
\boldsymbol{V}_{S_{\tau}}=\boldsymbol{\nu}_{S}\times(\boldsymbol{V}
\times\boldsymbol{\nu}_{S}),\quad
\boldsymbol{V}_{S_{\nu}}=\boldsymbol{\nu}_{S}\cdot\boldsymbol{V},
\end{eqnarray*}
which are the tangential and the normal components of a vector $\boldsymbol{V}$
on the boundary $\Gamma$ and $S$, respectively. It is clear to note that
$\boldsymbol{V}=\boldsymbol{V}_{\Gamma_{\tau}}+\boldsymbol{V}_{\Gamma_{\nu}}
\boldsymbol{\nu}_{\Gamma}$ on $\Gamma$
and
$\boldsymbol{V}=\boldsymbol{V}_{S_{\tau}}+\boldsymbol{V}_{S_{\nu}}\boldsymbol{
\nu}_{S}$ on $S$.
Denote by $\nabla_{\Gamma_{\tau}}$ and $\nabla_{S_{\tau}}$ the surface gradient
on $\Gamma$
and $S$, and denote by $\partial_{\boldsymbol{\nu}_{\Gamma}}$ and
$\partial_{\boldsymbol{\nu}_{S}}$ the normal derivative on $\Gamma$
and $S$, respectively.

Define the jump
\begin{eqnarray}\label{ET5-23}
[\boldsymbol{E}]
=\lim_{\substack{\boldsymbol{a}_{1}\to 0\\
\boldsymbol{x}+\boldsymbol{a}_{1}\in
\Omega_{1}}}\boldsymbol{E}_{1}(\boldsymbol{x}+\boldsymbol{a}_{1})
-\lim_{\substack{\boldsymbol{a}_{2}\to 0\\
\boldsymbol{x}+\boldsymbol{a}_{2}\in
\Omega_{2}}}\boldsymbol{E}_{2}(\boldsymbol{x}+\boldsymbol{a}_{2})
,\quad\boldsymbol{x}\in S,
\end{eqnarray}
of the continuous extension of a function $\boldsymbol{E}$ to the boundary from
$\Omega_{1}$ and $\Omega_{2}$, respectively.

\begin{theorem}\label{ET5}
Let $(\boldsymbol{E}_{1}, \boldsymbol{E}_{2})$ be the solutions of Problem
\ref{SP}.
Given $\boldsymbol{p}\in C^{2}(\Gamma\cup S, \mathbb{R}^{3})$,
the domain derivatives $(\boldsymbol{E}'_1, \boldsymbol E'_2)$ of
$(\boldsymbol{E}_1, \boldsymbol E_2)$ are the radiation solutions of the
following problem:
\begin{equation}\label{E1F}
\begin{cases}
\nabla\times\nabla\times\boldsymbol{E}'_{1}
-\kappa_{1}^{2}\boldsymbol{E}'_{1}=0 &\quad \text{in}\ \Omega_{1},\\
\nabla\times\nabla\times\boldsymbol{E}'_{2}
-\kappa_{2}^{2}\boldsymbol{E}'_{2}=0 &\quad \text{in}\ \Omega_{2},\\
\boldsymbol{\nu}_{\Gamma}\times{\boldsymbol{E}'_{1}}
=\big[\boldsymbol{p}_{\Gamma_{\nu}}(\partial_{\boldsymbol{\nu}_{\Gamma}}
\boldsymbol{E}_{1,\Gamma_{\tau}})+
\boldsymbol{E}_{1,\Gamma_{\nu}}(\nabla_{\Gamma_{\tau}}\boldsymbol{p}_{\Gamma_{
\nu}})\big]
\times\boldsymbol{\nu}_{\Gamma} &\quad \text{on}\ \Gamma,\\
[\boldsymbol{\nu}_{S}\times\boldsymbol{E}']
=-\mathrm{i}\omega[\mu\boldsymbol{H}_{S_{\tau}}]\boldsymbol{p}_{S_{\nu}}
-[\boldsymbol{\nu}_{S}\times(\nabla_{S_{\tau}}(\boldsymbol{p}_{S_{\nu}}
\boldsymbol{E}_{S_{\nu}}))] &\quad \text{on}\ S,\\
[\boldsymbol{\nu}_{S}\times\boldsymbol{H}']
=\mathrm{i}\omega[(\varepsilon+\mathrm{i}\frac{\sigma}{\omega})\boldsymbol{E}_{
S_{\tau}}]\boldsymbol{p}_{S_{\nu}}
-[\boldsymbol{\nu}_{S}\times(\nabla_{S_{\tau}}(\boldsymbol{p}_{S_{\nu}}
\boldsymbol{H}_{S_{\nu}}))] &\quad \text{on}\ S.
\end{cases}
\end{equation}
\end{theorem}

\begin{proof}
Define the operator $
\mathcal{A}=\nabla\times(\nabla\times)-\kappa_{1}^{2}\mathcal{I}$ and let
\begin{equation}\label{ET5-2}
\boldsymbol{\omega}_{\theta}=\mathcal{A}\boldsymbol{E}_{1, \theta},
\end{equation}
where $\boldsymbol{E}_{j,\theta}$ is a solution of Problem \ref{SP}
corresponding to the region $\Omega_{j,\theta}, j=1, 2$ for sufficiently small
$\theta$. Then, we have
\begin{equation}\label{ET5-3}
\boldsymbol{\omega}_{\theta}
=\boldsymbol{q}\delta\quad\text{in}\ \Omega_{1,\theta}
\end{equation}
and
\begin{equation}\label{ET5-4}
\boldsymbol{\omega}_{\theta}(\mathcal{I}+\theta)
=\boldsymbol{q}\delta\quad \text{in}\ \Omega_{1}.
\end{equation}

Since $\mathcal{A}$ is a linear and continuous operator from $H({\rm curl},
\Omega_{1})=\{\boldsymbol u\in L^2(\Omega_1)^3: \nabla\times\boldsymbol u\in
L^2(\Omega_1)^3\}$ into $\mathcal{D}'(\Omega_{1})$,
$\mathcal{A}$ is differentiable in the distribution sense, i.e.,
$\boldsymbol{\upsilon}\mapsto\langle\mathcal{A}\boldsymbol{\upsilon},
\boldsymbol{\psi}\rangle$ is differentiable for each
$\boldsymbol{\psi}\in\mathcal{D}(\Omega_{1})$ and
\begin{equation}\label{ET5-5}
\frac{\partial \mathcal{A}}{\partial\boldsymbol{\upsilon}}=\mathcal{A}.
\end{equation}
Here $\mathcal{D}(\Omega_{1})$ is the standard space of infinitely
differentiable functions with compact support in $\Omega_{1}$
and $\mathcal{D}'(\Omega_{1})$ is the standard space of distributions.
Therefore, it follows from the differentiability of
$\theta\mapsto\ \boldsymbol{E}_{1,\theta}(\mathcal{I}+\theta)$
and $\theta\mapsto\ \boldsymbol{E}_{1,\theta}$ that
$\theta\mapsto\ \boldsymbol{\omega}_{\theta}(\mathcal{I}+\theta)$
is continuously Fr\'{e}chet differentiable at $\theta=0$
in the direction $\boldsymbol{p}\in C^{2}(\Gamma\cup S, \mathbb{R}^{3})$.
Moreover, for an admissible perturbation $\theta$, their derivatives satisfy
\begin{equation}\label{ET5-6}
\frac{\partial}{\partial
\theta}(\boldsymbol{\omega}_{\theta}(\mathcal{I}+\theta))(0)\boldsymbol{p}
=\frac{\partial\boldsymbol{\omega}_{\theta}}{\partial \theta}(0)\boldsymbol{p}
+(\boldsymbol{p}\cdot\nabla)\boldsymbol{\omega}\quad \text{in}\ \Omega_{1}.
\end{equation}
We deduce from \eqref{ET5-2}--\eqref{ET5-4} and \eqref{ET5-6} that
\begin{align}\label{ET5-7}
\frac{\partial\boldsymbol{\omega}_{\theta}}{\partial \theta}(0)\boldsymbol{p}
&=\frac{\partial \mathcal{A}}{\partial\boldsymbol{E}_{1, \theta}}
\frac{\partial\boldsymbol{E}_{1, \theta}}{\partial \theta}(0)\boldsymbol{p}
=\frac{\partial
\mathcal{A}}{\partial\boldsymbol{E}_{1}}\boldsymbol{E}'_{1}\nonumber\\
&=\frac{\partial}{\partial
\theta}(\boldsymbol{\omega}_{\theta}(\mathcal{I}+\theta))(0)\boldsymbol{p}
-(\boldsymbol{p}\cdot\nabla)\boldsymbol{\omega}\nonumber\\
&=(\boldsymbol{p}\cdot\nabla)\boldsymbol{q}\delta
-(\boldsymbol{p}\cdot\nabla)\boldsymbol{q}\delta=0\quad\text{in}\
\Omega_{1}.
\end{align}
It follows  from \eqref{ET5-5} and \eqref{ET5-7} that
\begin{equation*}\label{ET5-8}
\mathcal{A}\boldsymbol{E}'_{1}=\nabla\times(\nabla\times\boldsymbol{E}'_{1})
-\kappa_{1}^{2}\boldsymbol{E}'_{1}=0\quad \text{in}\ \Omega_{1}.
\end{equation*}
For the boundary condition, we may follow the same steps as those in
\cite{L-JDE12} and obtain
\begin{equation*}\label{ET5-10}
\boldsymbol{\nu}_{\Gamma}\times{\boldsymbol{E}'_{1}}
=\big[\boldsymbol{p}_{\Gamma_{\nu}}(\partial_{\boldsymbol{\nu}_{\Gamma}}
\boldsymbol{E}_{1,\Gamma_{\tau}})+
\boldsymbol{E}_{1,\Gamma_{\nu}}(\nabla_{\Gamma_{\tau}}\boldsymbol{p}_{\Gamma_{
\nu}})\big]
\times\boldsymbol{\nu}_{\Gamma}\quad \text{on}\ \Gamma.
\end{equation*}

Furthermore, for every perturbation $\theta\in C^{2}(\Gamma\cup S,
\mathbb{R}^{3})$,
the tangential traces of the electric fields are assumed to be continuous across
$S$, i.e.,
\begin{eqnarray}\label{ET5-11}
\boldsymbol{\nu}_{\theta}\times\boldsymbol{E}_{1,\theta}
=\boldsymbol{\nu}_{\theta}\times\boldsymbol{E}_{2,\theta}\quad\text{on
}\ S_{\theta}.
\end{eqnarray}
Hence, we have
\begin{eqnarray}\label{ET5-12}
[\boldsymbol{\nu}_{\theta}(\mathcal{I}+\theta)]\times[\boldsymbol{E}_{1,
\theta}(\mathcal{I}+\theta)]
=[\boldsymbol{\nu}_{\theta}(\mathcal{I}+\theta)]\times[\boldsymbol{E}_{2,
\theta}(\mathcal{I}+\theta)]\quad\text{on}\ S.
\end{eqnarray}
Moreover, it follows from \cite[Lemma 3]{DF-JMAA99} and \cite[Lemma 4.8]{MS-74}
that
\begin{eqnarray}\label{ET5-13}
\boldsymbol{\nu}_{\theta}(\mathcal{I}+\theta)
=\frac{1}{\|g(\theta)\boldsymbol{\nu}_{S}\|_{L^{2}(S)}}g(\theta)\boldsymbol{\nu}
_{S}\quad\text{on}\ S,
\end{eqnarray}
where the matrix $g(\theta)=(\boldsymbol{I}+\frac{\partial\theta}{\partial \boldsymbol{x}})^{-\top}$
satisfies
\[
g(\boldsymbol{0})=\boldsymbol{I},\quad
\frac{\partial g(\theta)}{\partial \theta}(0)\boldsymbol{p}
=-(\nabla\boldsymbol{p})^{\top}.
\]

By \eqref{ET5-12} and \eqref{ET5-13}, we have
\begin{eqnarray}\label{ET5-15}
[g(\theta)\boldsymbol{\nu}_{S}]\times[\boldsymbol{E}_{1,\theta}(\mathcal{I}
+\theta)]
=[g(\theta)\boldsymbol{\nu}_{S}]\times[\boldsymbol{E}_{2,\theta}(\mathcal{I}
+\theta)]\quad\text{on}\ S
\end{eqnarray}
and
\begin{align}\label{ET5-16}
&\frac{\partial}{\partial
\theta}\{[g(\theta)\boldsymbol{\nu}_{S}]\times[\boldsymbol{E}_{1,\theta}
(\mathcal{I}+\theta)]\}(0)\boldsymbol{p}\notag\\
&=\frac{\partial}{\partial
\theta}\{[g(\theta)\boldsymbol{\nu}_{S}]\times[\boldsymbol{E}_{2,\theta}
(\mathcal{I}+\theta)]\}(0)\boldsymbol{p}\quad\text{on}\ S.
\end{align}

Using the chain rule, we deduce from \eqref{ET5-16} that
\begin{align}\label{ET5-17}
&\frac{\partial}{\partial
\theta}\{[g(\theta)\boldsymbol{\nu}_{S}]\times[\boldsymbol{E}_{j,\theta}
(\mathcal{I}+\theta)]\}(0)\boldsymbol{p}\nonumber\\
&=\bigg[\left(\frac{\partial g(\theta)}{\partial
\theta}(0)\boldsymbol{p}\right)\boldsymbol{\nu}_{S}\bigg]\times\boldsymbol{E}_{j
}
+\boldsymbol{\nu}_{S}\times\bigg[\frac{\partial}{\partial
\theta}(\boldsymbol{E}_{j,
\theta}(\mathcal{I}+\theta))(0)\boldsymbol{p}\bigg]\nonumber\\
&=-((\nabla\boldsymbol{p})^{\top}\boldsymbol{\nu}_{S})\times\boldsymbol{E}_{j}
+\boldsymbol{\nu}_{S}\times[\boldsymbol{E}'_{j}+(\boldsymbol{p}
\cdot\nabla)\boldsymbol{E}_{j}]
\quad\text{on}\ S,\ \ \ j=1,2.
\end{align}
Since on $S$ we have
\begin{align}\label{ET5-18}
((\nabla\boldsymbol{p})^{\top}\boldsymbol{\nu}_{S})\times\boldsymbol{E}_{j}
&=[\boldsymbol{\nu}_{S}\times(\nabla\times\boldsymbol{p})
+(\boldsymbol{\nu}_{S}\cdot\nabla)\boldsymbol{p}]\times\boldsymbol{E}_{j}
\nonumber\\
&=[\boldsymbol{\nu}_{S}\times(\nabla\times\boldsymbol{p})]\times\boldsymbol{E}_{
j}
+[(\boldsymbol{\nu}_{S}\cdot\nabla)\boldsymbol{p}]\times\boldsymbol{E}_{j}
\nonumber\\
&=-\boldsymbol{\nu}_{S}\times[\boldsymbol{E}_{j}\times(\nabla\times\boldsymbol{p
})]
-(\nabla\times\boldsymbol{p})\times(\boldsymbol{\nu}_{S}\times\boldsymbol{E}_{j}
)
+[(\boldsymbol{\nu}_{S}\cdot\nabla)\boldsymbol{p}]\times\boldsymbol{E}_{j}
\nonumber\\
&=-\boldsymbol{\nu}_{S}\times[\boldsymbol{E}_{j}\times(\nabla\times\boldsymbol{p
})]
-\boldsymbol{\nu}_{S}\times[(\boldsymbol{E}_{j}\cdot\nabla)\boldsymbol{p}]
\nonumber\\
&\quad-(\nabla\boldsymbol{p})(\boldsymbol{\nu}_{S}\times\boldsymbol{E}_{j})
+(\nabla\cdot\boldsymbol{p})(\boldsymbol{\nu}_{S}\times\boldsymbol{E}_{j})
,\quad j=1,2.
\end{align}
With the aid of \eqref{ET5-17} and \eqref{ET5-18}, we obtain
\begin{align}\label{ET5-19}
&\frac{\partial}{\partial \theta}\{[g(\theta)\boldsymbol{\nu}_{S}]\times[
\boldsymbol{E}_{j,\theta} (\mathcal{I}+\theta)]\}(0)\boldsymbol{p}\nonumber\\
&=-\{-\boldsymbol{\nu}_{S}\times[\boldsymbol{E}_{j}
\times(\nabla\times\boldsymbol{p})]
-\boldsymbol{\nu}_{S}\times[(\boldsymbol{E}_{ j}\cdot\nabla)\boldsymbol{p}]
-(\nabla\boldsymbol{p})(\boldsymbol{\nu}_{S}\times\boldsymbol{E}_{j})\notag\\
&\qquad+(\nabla\cdot\boldsymbol{p})(\boldsymbol{\nu}_{S}\times\boldsymbol{E}_{j}
)\}+\boldsymbol{\nu}_{S}\times\boldsymbol{E}'_{j}+\boldsymbol{\nu}_{S}\times
[(\boldsymbol{p}\cdot\nabla)\boldsymbol{E}_{j}]\nonumber\\
&=\{\boldsymbol{\nu}_{S}\times[\boldsymbol{E}_{j}\times(\nabla\times\boldsymbol{
p})]
+\boldsymbol{\nu}_{S}\times[(\boldsymbol{E}_{j}\cdot\nabla)\boldsymbol{p}]
+\boldsymbol{\nu}_{S}\times[(\boldsymbol{p}\cdot\nabla)\boldsymbol{E}_{j}]\}
\nonumber\\
&\qquad+\boldsymbol{\nu}_{S}\times\boldsymbol{E}'_{j}
+(\nabla\boldsymbol{p})(\boldsymbol{\nu}_{S}\times\boldsymbol{E}_{j})
-(\nabla\cdot\boldsymbol{p})(\boldsymbol{\nu}_{S}\times\boldsymbol{E}_{j}
)\nonumber\\
&=\boldsymbol{\nu}_{S}\times[\boldsymbol{E}_{j}\times(\nabla\times\boldsymbol{p}
)+(\boldsymbol{E}_{j}\cdot\nabla)\boldsymbol{p}
+(\boldsymbol{p}\cdot\nabla)\boldsymbol{E}_{j}]\nonumber\\
&\qquad+\boldsymbol{\nu}_{S}\times\boldsymbol{E}'_{j}
+(\nabla\boldsymbol{p})(\boldsymbol{\nu}_{S}\times\boldsymbol{E}_{j})
-(\nabla\cdot\boldsymbol{p})(\boldsymbol{\nu}_{S}\times\boldsymbol{E}_{j}
)\nonumber\\
&=\boldsymbol{\nu}_{S}\times[(\nabla\times\boldsymbol{E}_{j})\times\boldsymbol{p
}]+\boldsymbol{\nu}_{S}\times[\boldsymbol{p}
\times(\nabla\times\boldsymbol{E}_{j})
+\boldsymbol{E}_{j}\times(\nabla\times\boldsymbol{p})+(\boldsymbol{E}_{j
}\cdot\nabla)\boldsymbol{p}\notag\\
&\qquad+(\boldsymbol{p}\cdot\nabla)\boldsymbol{E}_{j}]+\boldsymbol{\nu}_{S}
\times\boldsymbol{E}'_{j}
+(\nabla\boldsymbol{p})(\boldsymbol{\nu}_{S}\times\boldsymbol{E}_{j})
-(\nabla\cdot\boldsymbol{p})(\boldsymbol{\nu}_{S}\times\boldsymbol{E}_{j}
)\nonumber\\
&=\boldsymbol{\nu}_{S}\times[(\nabla\times\boldsymbol{E}_{j})\times\boldsymbol{p
}]+\boldsymbol{\nu}_{S}\times[\nabla(\boldsymbol{p}
\cdot\boldsymbol{E}_{j})]
+\boldsymbol{\nu}_{S}\times\boldsymbol{E}'_{j}\notag\\
&\qquad+(\nabla\boldsymbol{p})(\boldsymbol{\nu}_{S}\times\boldsymbol{E}_{j})
-(\nabla\cdot\boldsymbol{p})(\boldsymbol{\nu}_{S}\times\boldsymbol{E}_{j}
)\nonumber\\
&=\mathrm{i}\omega[\boldsymbol{\nu}_{S}\times((\mu\boldsymbol{H}_{j}
)\times\boldsymbol{p})]
+[\boldsymbol{\nu}_{S}\times(\nabla(\boldsymbol{p}\cdot\boldsymbol{E}_{j}))]
\nonumber\\
&\qquad+\boldsymbol{\nu}_{S}\times\boldsymbol{E}'_{j}
+(\nabla\boldsymbol{p})(\boldsymbol{\nu}_{S}\times\boldsymbol{E}_{j})
-(\nabla\cdot\boldsymbol{p})(\boldsymbol{\nu}_{S}\times\boldsymbol{E}_{j})
\quad\text{on}\ S,\ j=1,2.
\end{align}

By taking into account of the continuous conditions \eqref{EBCS} and
$\boldsymbol{p}\in C^{2}(\Gamma\cup S, \mathbb{R}^{3})$,
from \eqref{ET5-23} and \eqref{ET5-20}, the jump relations read
\begin{align}\label{ET5-20}
[\boldsymbol{\nu}_{S}\times\boldsymbol{E}']
=-\mathrm{i}\omega[\boldsymbol{\nu}_{S}\times((\mu\boldsymbol{H}
)\times\boldsymbol{p})]
-[\boldsymbol{\nu}_{S}\times(\nabla(\boldsymbol{p}\cdot\boldsymbol{E}))].
\end{align}

For the first term of in the right hand side of \eqref{ET5-20}, we conclude
from the jump condition $[\mu\boldsymbol{H}_{S_{\nu}}]=0$ that
\begin{align}\label{ET5-21}
\mathrm{i}\omega[\boldsymbol{\nu}_{S}\times((\mu\boldsymbol{H})\times\boldsymbol
{p})]
&=\mathrm{i}\omega[(\mu\boldsymbol{H})(\boldsymbol{\nu}_{S}\cdot\boldsymbol{p})
-\boldsymbol{p}(\boldsymbol{\nu}_{S}\cdot(\mu\boldsymbol{H}))]\nonumber\\
&=\mathrm{i}\omega[\mu(\boldsymbol{H}_{S_{\tau}}+\boldsymbol{H}_{S_{\nu}}
\boldsymbol{\nu}_{S})\boldsymbol{p}_{S_{\nu}}
-(\boldsymbol{p}_{S_{\tau}}+\boldsymbol{p}_{S_{\nu}}\boldsymbol{\nu}_{S}
)(\mu\boldsymbol{H}_{S_{\nu}})]\nonumber\\
&=\mathrm{i}\omega[\mu\boldsymbol{H}_{S_{\tau}}\boldsymbol{p}_{S_{\nu}}
-\mu\boldsymbol{H}_{S_{\nu}}\boldsymbol{p}_{S_{\tau}}]\nonumber\\
&=\mathrm{i}\omega[\mu\boldsymbol{H}_{S_{\tau}}]\boldsymbol{p}_{S_{\nu}}
-\mathrm{i}\omega[\mu\boldsymbol{H}_{S_{\nu}}]\boldsymbol{p}_{S_{\tau}}
\nonumber\\
&=\mathrm{i}\omega[\mu\boldsymbol{H}_{S_{\tau}}]\boldsymbol{p}_{S_{\nu}}
\quad\text{on}\ S.
\end{align}
It follows from
$[\boldsymbol{\nu}_{S}\times\boldsymbol{E}]=[\boldsymbol{\nu}_{S}
\times\boldsymbol{E}_{S_{\tau}}]=0$
and the definition of the surface gradient $\nabla_{S_{\tau}}$ that
we obtain
$[\boldsymbol{\nu}_{S}\times(\nabla_{S_{\tau}}(\boldsymbol{p}_{S_{\tau}}
\cdot\boldsymbol{E}_{S_{\tau}}))]=0$.
Thus, the second term in the right hand side of \eqref{ET5-20} reduces to
\begin{align}\label{ET5-22}
[\boldsymbol{\nu}_{S}\times(\nabla(\boldsymbol{p}\cdot\boldsymbol{E}))]
&=[\boldsymbol{\nu}_{S}\times(\nabla_{S_{\tau}}(\boldsymbol{p}\cdot\boldsymbol{E
}))]\nonumber\\
&=[\boldsymbol{\nu}_{S}\times(\nabla_{S_{\tau}}((\boldsymbol{p}_{S_{\tau}}
+\boldsymbol{p}_{S_{\nu}}\boldsymbol{\nu}_{S})
\cdot(\boldsymbol{E}_{S_{\tau}}+\boldsymbol{E}_{S_{\nu}}\boldsymbol{\nu}_{S})))]\nonumber\\
&=[\boldsymbol{\nu}_{S}\times(\nabla_{S_{\tau}}(\boldsymbol{p}_{S_{\tau}}
\cdot\boldsymbol{E}_{S_{\tau}}
+\boldsymbol{p}_{S_{\nu}}\boldsymbol{E}_{S_{\nu}}))]\nonumber\\
&=[\boldsymbol{\nu}_{S}\times(\nabla_{S_{\tau}}(\boldsymbol{p}_{S_{\nu}}
\boldsymbol{E}_{S_{\nu}}))]
\quad \text{on}\ S.
\end{align}
Finally, by \eqref{ET5-20}--\eqref{ET5-22}, we have the boundary condition
\begin{eqnarray*}\label{ET5-24}
[\boldsymbol{\nu}_{S}\times\boldsymbol{E}']
=-\mathrm{i}\omega[\mu\boldsymbol{H}_{S_{\tau}}]\boldsymbol{p}_{S_{\nu}}
-[\boldsymbol{\nu}_{S}\times(\nabla_{S_{\tau}}(\boldsymbol{p}_{S_{\nu}}
\boldsymbol{E}_{S_{\nu}}))]\quad \text{on}\ S.
\end{eqnarray*}
Similarly, we can obtain
\begin{eqnarray*}\label{ET5-25}
[\boldsymbol{\nu}_{S}\times\boldsymbol{H}']
=\mathrm{i}\omega[(\varepsilon+\mathrm{i}\frac{\sigma}{\omega})\boldsymbol{E}_{
S_{\tau}}]\boldsymbol{p}_{S_{\nu}}
-[\boldsymbol{\nu}_{S}\times(\nabla_{S_{\tau}}(\boldsymbol{p}_{S_{\nu}}
\boldsymbol{H}_{S_{\nu}}))]\quad \text{on}\ S.
\end{eqnarray*}

Based on the existence of the domain derivatives $\boldsymbol{E}'_{j}$,
the proof of the the integral representations for $\boldsymbol{E}'_{j}$
follow in the same manner as for the the integral representation of
$\boldsymbol{E}_{j}$. Therefore, the asymptotic behavior to the domain
derivative $\boldsymbol{E}'_{j}$ has the same form as
$\boldsymbol{E}_{j}$. This means that the domain derivatives
$(\boldsymbol{E}'_1, \boldsymbol{E}'_2)$ are the radiation solutions of the
problem \eqref{E1F}.
\end{proof}

Introduce the domain $\Omega_{1,h}$ bounded by $\Gamma_{h}$ and $S_{h}$,
where
\begin{eqnarray*}\label{GaSh}
\Gamma_{h}=\{\boldsymbol{x}+hp(\boldsymbol{x})\boldsymbol{\nu}_{\Gamma}
:\boldsymbol{x}\in \Gamma\},\quad
S_{h}=\{\boldsymbol{x}+hp(\boldsymbol{x})\boldsymbol{\nu}_{S}:\boldsymbol{x}\in
S\}.
\end{eqnarray*}
where $p\in C^{2}(\mathbb{R}^{3}, \mathbb{R})$ and $h>0$.
For any two domains $\Omega_{1}$ and $\Omega_{1,h}$ in $\mathbb{R}^{3}$, define
the Hausdorff distance
\begin{eqnarray*}\label{Om12}
\mathrm{dist }(\Omega_{1} ,\Omega_{1,h})=\max\{\rho(\Omega_{1,h} ,\Omega_{1}),
\rho(\Omega_{1} ,\Omega_{1,h})\},
\end{eqnarray*}
where
\begin{eqnarray*}
\rho(\Omega_{1}
,\Omega_{1,h})=\sup_{\boldsymbol{x}\in\Omega_{1}}\inf_{\boldsymbol{y}\in\Omega_{
1,h}}|\boldsymbol{x}-\boldsymbol{y}|.
\end{eqnarray*}

It can be easily seen that the Hausdorff distance between $\Omega_{1,h}$ and
$\Omega_{1}$ is of the order $h$,
i.e., $\mathrm{dist} (\Omega_{1} ,\Omega_{1,h}) =\mathcal O(h)$.
We have the following local stability result.

\begin{theorem}\label{ET6}
If $p\in C^{2}(\Gamma\cup S, \mathbb{R})$ and $h > 0$ is sufficiently small,
then
\begin{eqnarray*}\label{dist}
\mathrm{dist }(\Omega_{1} ,\Omega_{1,h})\leq
C\|
\boldsymbol{\nu}_{\Gamma_{H}}\times\boldsymbol{E}_{1,h}-\boldsymbol{\nu}_{
\Gamma_{H}}\times \boldsymbol{E}_{1}\|_{C^{0,\alpha}(\Gamma_{H})},
\end{eqnarray*}
where $\boldsymbol{E}_{1,h}$ and $\boldsymbol{E}_{1}$ is the solution of
Problem \ref{SP} corresponding to the domain $\Omega_{1, h}$ and $\Omega_1$,
respectively, and $C$ is a positive constant independent of $h$.
\end{theorem}

\begin{proof}
Assume by contradiction that there exists a subsequence from
$\{\boldsymbol{E}_{1,h}\}$, which is still denoted as $\{\boldsymbol{E}_{1,h}\}$
for simplicity, such that
\begin{eqnarray*}
\lim_{h\to 0}\bigg\|
\frac{\boldsymbol{\nu}_{\Gamma_{H}}\times\boldsymbol{E}_{1,h}-\boldsymbol{\nu}_{
\Gamma_{H}}\times \boldsymbol{E}_{1}}{h}\bigg\|_{C^{0,\alpha}(\Gamma_{H})}
=\|\boldsymbol{\nu}_{\Gamma_{H}}\times\boldsymbol{E}'_{1}\|_{C^{0,\alpha}
(\Gamma_{H})} =0\quad \text{as}\ h\to 0,
\end{eqnarray*}
which yields $\boldsymbol{\nu}_{\Gamma_{H}}\times\boldsymbol{E}'_{1}=0$ on
$\Gamma_{H}$.
Following a similar proof of Theorem \ref{ET3}, we can show the uniqueness of
the solution for problem \eqref{E1F}.
An application of the uniqueness for problem \eqref{E1F} yields that
$\boldsymbol{E}'_{j}= 0$ in $\Omega_{j}, j=1, 2$.
Noting  the boundary condition of $\boldsymbol{E}'_{1}$ in problem \eqref{E1F}
gives
\begin{align}\label{ET6-2}
\boldsymbol{\nu}_{\Gamma}\times{\boldsymbol{E}'_{1}}
&=[(p(\boldsymbol{x})\boldsymbol{\nu}_{\Gamma})_{\Gamma_{\nu}}(\partial_{
\boldsymbol{\nu}_{\Gamma}}\boldsymbol{E}_{1,\Gamma_{\tau}})+
\boldsymbol{E}_{1,\Gamma_{\nu}}(\nabla_{\Gamma_{\tau}}(p(\boldsymbol{x}
)\boldsymbol{\nu}_{\Gamma})_{\Gamma_{\nu}})]\times\boldsymbol{\nu}_{\Gamma}
\nonumber\\
&=[p(\partial_{\boldsymbol{\nu}_{\Gamma}}\boldsymbol{E}_{1,\Gamma_{\tau}})+
\boldsymbol{E}_{1,\Gamma_{\nu}}(\nabla_{\Gamma_{\tau}}p)]\times\boldsymbol{\nu}
_{\Gamma}=0
\quad \text{on}\ \Gamma.
\end{align}
Since $p$ is arbitrary in \eqref{ET6-2}, we have
\begin{align}\label{ET6-3}
\partial_{\boldsymbol{\nu}_{\Gamma}}\boldsymbol{E}_{1,\Gamma_{\tau}}
&=\partial_{\boldsymbol{\nu}_{\Gamma}}[\boldsymbol{\nu}_{\Gamma}
\times(\boldsymbol{E}_{1}\times\boldsymbol{\nu}_{\Gamma})]\nonumber\\
&=\partial_{\boldsymbol{\nu}_{\Gamma}}\boldsymbol{E}_{1}
-\partial_{\boldsymbol{\nu}_{\Gamma}}[(\boldsymbol{\nu}_{
\Gamma}\cdot\boldsymbol{E}_{1})\boldsymbol{\nu}_{\Gamma}] =0\quad\text{on}\
\Gamma
\end{align}
and
\begin{eqnarray}\label{ET6-4}
\boldsymbol{E}_{1,\Gamma_{\nu}}=\boldsymbol{\nu}_{\Gamma}\cdot\boldsymbol{E}_{1}
=0
\quad \text{on}\ \Gamma.
\end{eqnarray}
It follows from \eqref{ET6-3} and \eqref{ET6-4} that
\begin{eqnarray}\label{ET6-5}
\partial_{\boldsymbol{\nu}_{\Gamma}}\boldsymbol{E}_{1}=0\quad
\text{on}\ \Gamma.
\end{eqnarray}
With the aid of $\boldsymbol{\nu}_{\Gamma}\times\boldsymbol{E}_{1}|_{\Gamma} =0$
and $\boldsymbol{\nu}_{\Gamma}\cdot\boldsymbol{E}_{1}|_{\Gamma}=0$,
we have
\begin{eqnarray}\label{ET6-6}
\boldsymbol{E}_{1} =0\quad \text{on}\ \Gamma.
\end{eqnarray}
Therefore, combining \eqref{ET6-5} and \eqref{ET6-6}, we infer by unique
continuation that
 \begin{eqnarray*}\label{ET6-7}
\boldsymbol{E}_{1} =0 \quad \text{in}\ \Omega_{1},
\end{eqnarray*}
which is a contradiction to the
 \begin{eqnarray*}\label{ET6-8}
\nabla\times(\nabla\times\boldsymbol{E}_{1})-\kappa_{1}^{2}\boldsymbol{E}_{1}
=\mathrm{i}\omega\mu_{1}\boldsymbol{J}_{cs}\neq0 \quad\text{in}\
\Omega_{1}.
\end{eqnarray*}
The proof is completed.
\end{proof}

\section{Conclusion}\label{cl}

In this paper, we have studied the direct and inverse electromagnetic obstacle
scattering problems for the three-dimensional Maxwell equations in an unbounded
structure. We present an equivalent integral equation to the boundary value
problem and show that it has a unique solution. For the
inverse problem, we prove that the obstacle and unbounded rough surface can be
uniquely determined by the tangential component of the electric field measured
on the plane surface above the obstacle. The local stability shows
that the Hausdorff distance of the two regions, corresponding to small
perturbations of the obstacle and the unbounded rough surface, is bounded
by the distance of corresponding tangential trace of the electric fields if they
are close enough. To prove the stability, the domain derivative of the electric
field with respect to the change of the shape of the obstacle and
unbounded rough surface is examined. In particular, we deduce that the domain
derivative satisfies a boundary value problem of the Maxwell equations, which is
similar to the model equation of the direct problem.

\end{document}